\documentclass[a4paper,11pt]{article}
\usepackage[T1]{fontenc}
\usepackage{lmodern}
\usepackage{latexsym}
\usepackage{amsfonts}
\usepackage{amssymb}
\usepackage{mathtools}
\usepackage{amsthm}
\usepackage{multirow}
\usepackage[pdftex]{graphicx}
\usepackage{wrapfig}
\usepackage{float}
\theoremstyle{definition}
\newtheorem{theorem}{Theorem}

\newtheorem{lemma}[theorem]{Lemma}

\newtheorem{corollary}[theorem]{Corollary}
\newtheorem{claim}{Claim}
\newtheorem{reftheorem}{Theorem}

\newtheorem{problem}{Problem}

\newcommand{\Cforw}{{\overrightarrow{C}}}
\newcommand{\Cback}{{\overleftarrow{C}}}

\newcommand{\claimproof}{%
\smallbreak\noindent%
\emph{Proof.}%
}

\newcommand{\naturalnumbers}{{\boldsymbol{N}}}

\newcommand{\GG}{\mathcal{G}}

\newcommand{\HH}{\mathcal{H}}

\newcommand{\qHH}{\HH/\!\equiv}
\begin{document}
\title{Preorder Induced by Rainbow Forbidden Subgraphs}
\author{Shun-ichi Maezawa
\thanks{Department of Information Science, Nihon University, Sakurajosui 3--25--40, Setagaya-Ku Tokyo 156-8550, Japan.
\texttt{maezawa.mw@gmail.com, saitou.akira@nihon-u.ac.jp}}
\thanks{Research supported by JSPS KAKENHI Grant number JP20H05795 and JP22K13956}
\and Akira Saito
\footnotemark[1]
\thanks{Research supported by JSPS KAKENHI Grant number JP20K11684 and JP24K06835}
}
\date{}
\maketitle
\begin{abstract}
A subgraph $H$ of an edge-colored graph $G$
is rainbow if all the edges of $H$ receive different colors.
If $G$ does not contain a rainbow subgraph isomorphic to $H$,
we say that $G$ is rainbow
$H$-free.
For connected graphs $H_1$ and $H_2$,
if every rainbow $H_1$-free edge-colored complete graph
colored in sufficiently many colors is rainbow $H_2$-free,
we write $H_1\le H_2$.
The binary relation $\le$ is reflexive and transitive,
and hence it is a preorder.
If $H_1$ is a subgraph of $H_2$,
then trivially $H_1\le H_2$ holds.
On the other hand,
there exists a pair $(H_1, H_2)$ such that $H_1$ is a proper supergraph
of $H_2$ and $H_1\le H_2$ holds.
Cui et al.~[Discrete Math.~\textbf{344} (2021) Article Number 112267]
characterized these pairs.
In this paper,
we investigate the pairs $(H_1, H_2)$ with $H_1\le H_2$
when neither $H_1$ nor $H_2$ is a subgraph of the other.
We prove that there are many such pairs and investigate their structure with respect to $\le$.
\end{abstract}
\medbreak\noindent
\section{Introduction}
For graphs $H$ and $G$,
we say that $G$ is $H$-free
if $G$ does not contain an induced subgraph
isomorphic to $H$.
If $G$ is $H$-free,
we also say that $H$ is forbidden in $G$.
\par
For an appropriate choice of $H$ and a graph property $P$,
we can observe a phenomenon that
the class of $H$-free graphs is well-behaved
with respect to $P$.
The study of these combinations $(H, P)$ is called forbidden subgraphs.
It is one of the popular topics in graph theory.
\par
For the sake of the subsequent discussions,
we restate this topic
in a more formal manner.
In order to avoid set-theoretic ambiguity,
we only consider finite simple graphs and we take vertices
from the set of natural numbers $\naturalnumbers$.
Define $\HH_0$ by
\[
\HH_0 = \{H\colon \text{$H$ is a finite graph with $V(H)\subset\naturalnumbers$}\}.
\]
We sometimes consider the quotient set of $\HH_0$ with respect to
the equivalence relation based on the isomorphism.
But as in standard textbooks of graph theory,
we slightly abuse the notation and adopt the same symbol for
both $\HH_0$ and its quotient set.
\par
For $H_1, H_2\in\HH_0$,
we write $H_1\preceq H_2$ if $H_1$ is an induced subgraph of $H_2$.
It is easy to see that $\preceq$ is a partial order in $\HH_0$.
We introduce another binary relation based on
forbidden subgraphs.
We write $H_1\le_F H_2$ if there exists a constant
$t=t(H_1, H_2)$ such that every $H_1$-free graph
of order $t$ or more is $H_2$-free.
Trivially,
the binary relation $\le_F$ is reflexive and transitive,
and hence it is a preorder.
On the other hand,
it is not anti-symmetric.
Fujita,
Furuya and Ozeki~\cite{FFO}
gave a pair of non-isomorphic graphs $H_1$ and $H_2$
which satisfies both $H_1\le_F H_2$ and $H_2\le_F H_1$. 
\par
It is easy to see that $H_1\preceq H_2$ implies $H_1\le_F H_2$.
On the other hand,
in the case of $H_2\preceq H_1$,
we cannot expect $H_1\le_F H_2$
unless $H_1=H_2$.
Let $H_1$ and $H_2$ be connected graphs in $\HH_0$
and suppose $H_2$ is a proper induced subgraph of $H_1$,
then for every positive integer $n$,
the graph consisting of $n$ disjoint copies of $H_2$
is $H_1$-free,
but not $H_2$-free.
\par
In this paper,
we study an analogue of the above forbidden subgraphs
in edge-colored graphs.
For
a graph $G$,
we associate a mapping $c\colon E(G)\to\naturalnumbers$
and call the pair $(G, c)$ an \textit{edge-colored graph}.
For $e\in E(G)$,
we call $c(e)$ the \textit{color} of $e$.
Note that we do not require $c$ to be proper,
i.e.~$G$ may contain two different edges receiving the same color
incident with a common vertex.
An edge-colored graph $(G, c)$ is said to be \textit{rainbow}
if $c$ is injective,
i.e.~all the edges in $G$ receive different colors.
For $H\in \HH_0$,
we say that $(G, c)$ is \textit{rainbow $H$-free\/}
if $G$ does not contain a rainbow subgraph
which is isomorphic to $H$.
\par
As far as we know,
the study of forbidden rainbow subgraphs
was initiated by Gallai~\cite{Gallai}.
He introduced the notion of Gallai coloring,
which is a rainbow $C_3$-free edge-colored complete graph.
He gave a constructive characterization of Gallai coloring
and investigated its properties.
Gallai coloring is still actively studied,
and other rainbow forbidden subgraphs have also been studied since then.
For the readers who are interested in this subject,
we refer them to~\cite{MN2020}.
\par
In the study of rainbow forbidden subgraphs,
we often restrict ourselves to complete graphs
as a research domain,
as Gallai did for Gallai coloring.
One reason for this is that
for an edge-colored graph $(G, c)$,
we assign a new color to the edges of the complement of $G$,
and embed $(G, c)$ in an edge-colored complete graph.
In this paper,
we also restrict ourselves to edge-colored complete graphs.
\par
For $H_1, H_2\in\GG$,
we write $H_1\le H_2$ if
there exists a positive integer $t=t(H_1, H_2)$ such that
every rainbow $H_1$-free edge-colored complete graph
colored in $t$ or more colors
is rainbow $H_2$-free.
There is a rationale for the study of $\le$.
Suppose $H_1\le H_2$ and every rainbow $H_2$-free edge-colored complete
graph colored in sufficiently many colors
satisfies some graph property $P$.
Then every rainbow $H_1$-free edge-colored complete graph
colored in sufficiently many colors
also satisfies $P$.
Therefore,
the study of $P$ for rainbow $H_1$-free graphs means little.
In this sense,
the study of the relation $\le$ can
facilitate the research of rainbow forbidden subgraphs.
\par
For $H_1, H_2\in\HH_0$,
if $H_1$ is a subgraph of $H_2$,
we write $H_1\subseteq H_2$.
Note that $H_1\subseteq H_2$ does not mean
$H_1$ is an \lq\lq
induced'' subgraph of $H_2$.
If $H_1\subseteq H_2$ and $H_1\ne H_2$,
we write $H_1\subsetneq H_2$ and
we say that $H_2$ is a proper supergraph of $H_1$.
If $H_1\subseteq H_2$,
then trivially $H_1\le H_2$ holds.
This fact is similar to that of the traditional forbidden subgraphs
in the uncolored graphs.
However,
for rainbow forbidden subgraphs,
$H_1\le H_2$ may hold even if $H_1$ is a proper supergraph of $H_2$.
Bass et al.~\cite{BMOP} first reported this phenomenon.
Let $K^+_{1, k}$ be the graph obtained from $K_{1, k}$,
the star of order $k+1$,
by subdividing one edge with one vertex
(Figure~\ref{star_plus}).
\begin{figure}
\centering
\includegraphics[width=0.4\textwidth]{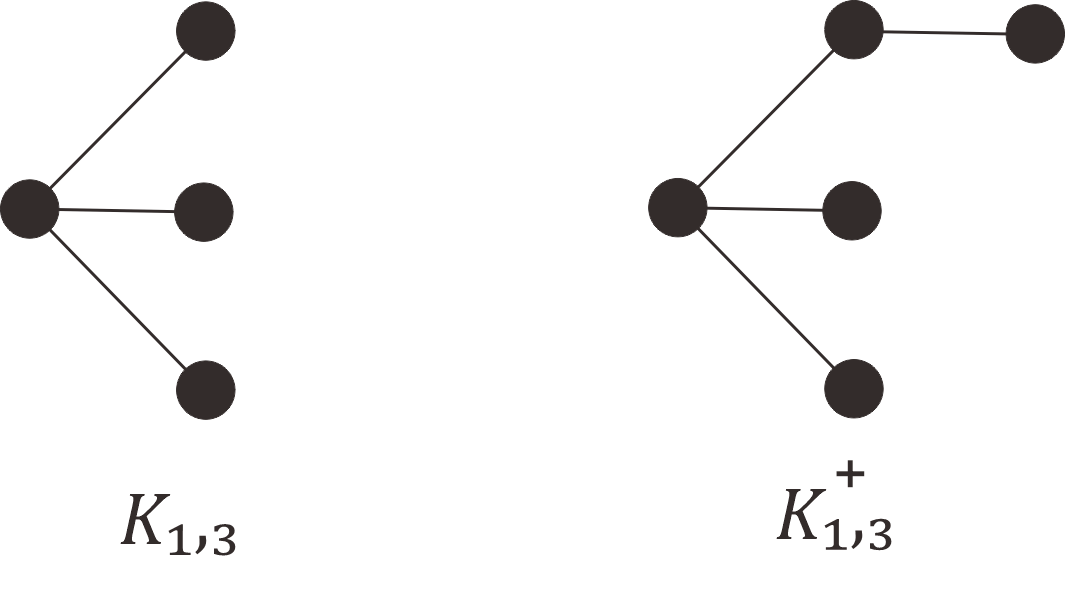}
\caption{$K_{1,3}$ and $K^+_{1,3}$}
\label{star_plus}
\end{figure}
They proved that every rainbow $K_{1,3}^+$-free edge-colored complete graph
colored in $5$ or more colors is rainbow $K_{1,3}$-free.
Cui et al.~\cite{CLMS} made a deeper analysis
and determined when it happens.
\begin{reftheorem}[\cite{CLMS}]
Let $H_1, H_2\in \HH_0$.
Suppose $|H_2|\ge 4$.
Then both $H_2\subsetneq H_1$ and $H_1\le H_2$ hold
if and only if $(H_1, H_2)=(K_{1,k}^+, K_{1,k})$
for some $k\ge 3$.
\label{clms}
\end{reftheorem}
\par
Theorem~\ref{clms} determines the pairs of graph $(H_1, H_2)$
with $H_1\le H_2$
when $H_1$ and $H_2$ are comparable with respect to $\subseteq$.
Motivated by this result,
in this paper,
we study the pairs of graphs $H_1$ and $H_2$
with $H_1\le H_2$
when neither $H_1$ nor $H_2$ is a subgraph of the other.
We will show that there are variety of such pairs and
they have noteworthy structure.
\par
In the next section,
we remark that not every connected graph in $\HH_0$
is eligible for the study,
and determine the precise domain of research.
In our study,
it is convenient to have invariants of graphs
which naturally reflect the comparability in $\le$.
We seek such invariants in Section~3.
In Section~4,
we give examples of a pair of graphs
$H_1$ and $H_2$ with $H_1\le H_2$,
neither of which is a subgraph of the other,
and in Section~5,
we study a partial order naturally induced
by the preorder $\le$.
In Section~6,
we make some concluding remarks and present future targets in this study.
\par
For terminology and notation
not explained in this paper,
we refer the reader to \cite{CLZ}.
As we mentioned earlier in this section,
the vertex set $V(G)$ of a graph $G$ is a finite subset of $\naturalnumbers$.
The degree of a vertex $x$ in $G$ is denoted by $\deg_G x$,
and the maximum degree of $G$ is denoted by $\Delta(G)$.
For edge-colored graphs $(G, c)$,
we often deal with color degrees.
The number of colors that appear in the edges incident with $x$
is denoted by $d_c(x)$ and called the color degree of $x$.
The maximum value among color degrees of the vertices in $G$
is called the maximum color degree and denoted by $\Delta_c(G,c)$.
For disjoint subsets $X$ and $Y$ of $V(G)$,
we denote by $E_G(X, Y)$ the set of edges
with one endvertex in $X$ and the other in $Y$.
\par
Given a set of vertices $X$,
let $K[X]$ denote the complete graph whose vertex set is $X$.
In this paper,
we always use
the symbol
\lq $K$'
to indicate that it it is a complete graph.
Therefore,
once the graph $K=K[X]$ is introduced,
we can interpret $K[Y]$ for $Y\subset X$
also as the subgraph induced by $Y$.
\par
In this paper,
we also consider graphs induced by a set of edges in this paper.
Let $G$ be a graph and let $F\subset E(G)$.
Then we define $G[F]$ to be the graph with
vertex set $V(F)$ and the edge set $F$,
where $V(F)$ is the set of endvertices of the edges in $F$.
If $F=\{f_1,\dots, f_k\}$,
we often write $G[f_1,\dots, f_k]$
instead of $G[\{f_1,\dots, f_k\}]$.
\par
In this paper,
a color of an edge is represented by a natural number,
and we often consider a set of colors in a certain range.
For this purpose,
we use the standard notation of a closed interval.
For $a, b\in\naturalnumbers$ with $a\le b$,
we denote by $[a, b]$ the set of natural numbers $n$ with
$a\le n\le b$.
\par
We denote by $P_k$ and $C_k$ the path and the cycle of order~$k$,
respectively.
For a cycle $C=x_1x_2\dots, x_kx_1$,
we denote by $x_i\Cforw x_j$ the subpath
$x_ix_{i+1}\dots, x_{j-1}x_j$.
The same path traversed in the opposite direction
is denoted by $x_j\Cback x_i$.
For a cycle or a path $S$,
we denote by $l(S)$ the number of edges in $S$
and call it the length of $S$.
\par
The \textit{barbell},
denoted by $B$,
is the graph obtained from $K_{1,3}^+$ by attaching a pendant edge
to the unique vertex of degree~$2$(Figure~\ref{barbell_figure}).
Equivalently,
the barbell is the unique tree having degree sequence $(3, 3, 1, 1, 1, 1)$.
\begin{figure}
\centering
\includegraphics[width=0.2\textwidth]{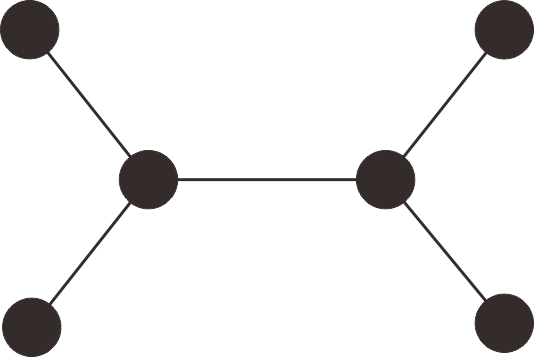}
\caption{barbell $B$}
\label{barbell_figure}
\end{figure}
%%%%%%%%%%%%%%%%%%%%%%%%%%%%%%%%%%%%%%%%%%%%%%%%%%%%%%%%%%%%%%
\section{Domain and Target of Study}
As we mentioned in the introduction,
we take vertices from the set of natural numbers
and consider only finite simple graphs.
Also we deal with only connected graphs in this paper,
though we do not think that forbidding a disconnected graph
is a trivial subject.
In the study of the traditional forbidden subgraphs
in the uncolored graphs,
the class of $2K_2$-free graphs has been reported to have a number
of intriguing cycle-related properties.
See \cite{CESS2017},
\cite{GP2016} and \cite{OS2022},
for example.
We exclude disconnected graphs in this paper
only because
we make sufficiently many observations
for connected graphs.
\par
Let $\HH_1$ be the set of connected graph in $\HH_0$.
We will study the relation $\le$ in $\HH_1$.
However,
there is no meaning in forbidding $K_1$ and $K_2$.
Hence we would like to exclude them from our study.
However,
we actually have to exclude more.
\par
In~\cite{TW2007},
Thomason and Wagner made the following observation.
\begin{reftheorem}[\cite{TW2007}]
Every edge-colored complete graph colored in $4$ or more colors
contains a rainbow $P_4$.
\label{tw}
\end{reftheorem}
By Theorem~\ref{tw},
whatever graph $H$ in $\HH_1$ we take,
we cannot give a counterexample
to the statement
\lq\lq
Every rainbow $P_4$-free edge-colored complete graph
colored in $5$ or more colors is rainbow $H$-free'',
and hence $P_4\le H$ vacuously holds.
However,
it does not give any insight. 
Therefore,
we may want to exclude $P_4$ and its subgraphs.
Considering the above,
we define $\HH$ by $\HH=\HH_1-\{P_1, P_2, P_3, P_4\}$,
and study the binary relation $\le$ in $\HH$.
Before proceeding further,
we remark that we do not have to weed out more.
\begin{lemma}
For every $H\in\HH$ and every positive integer $t$,
there exists a rainbow $H$-free edge-colored complete graph
colored in $t$ or more colors.
\label{domain}
\end{lemma}
\begin{proof}
Take a set of $2(t-1)$ vertices $X=\{x_1,\dots, x_{t-1}, y_1,\dots, y_{t-1}\}$
from $\naturalnumbers$ and let $K=K[X]$.
Define $c\colon E(K)\to\naturalnumbers$ by
\[
c(e) = 
\begin{cases}
i & \text{if $e=x_iy_i$, $1\le i\le t-1$}\\
t & \text{otherwise.}    
\end{cases}
\]
Then $(K, c)$ is a complete graph edge-colored in $t$ colors.
Moreover,
every connected rainbow subgraph in $(K, c)$ is a subgraph
of $P_4$,
and hence it is rainbow $H$-free.
\end{proof}
Next,
we discuss the binary relation $\le$.
As we have seen in the introduction,
it is reflective and transitive,
ans hence it is a preorder.
On the other hand,
we know
$K_{1,k}^+\le K_{1,k}$
by Theorem~\ref{clms},
and since $K_{1,k}\subseteq K_{1,k}^+$,
we also have $K_{1,k}\le K^+_{1,k}$.
Therefore,
$\le$ is not anti-symmetric and it is
not a partial order.
\par
In this situation,
there is a well-known approach to convert a preorder $\le$
to a partial order whose behavior is similar to $\le$.
For $H_1, H_2\in\HH$,
we write
$H_1\equiv H_2$ if both $H_1\le H_2$ and $H_2\le H_1$ hold.
It is easy to see that $\equiv$ is an equivalence relation.
For equivalence classes $\HH'$ and $\HH''$ with respect to $\equiv$,
we write $\HH'\le \HH''$ if
there exist $H_1\in\HH'$ and $H_2\in\HH''$
with $H_1\le H_2$.
Again,
it is easy to see that $\le$ is a partial order on the quotient set $\qHH$.
Note that in order to avoid unnecessary complication,
we slightly abuse the notation and use the same symbol $\le$
for both the preorder in $\HH$ and the partial order in $\qHH$.
\par
In Section~5,
we determine the minimum element of
$(\qHH, \le\,)$ and the elements
immediately after it.
%%%%%%%%%%%%%%%%%%%%%%%%%%%%%%%%%%%%%%%%%%%%%%%%%%%%%%%%%%%%%%
\section{Invariants of Graphs Reflecting $\boldsymbol{\le}$}
When we seek candidates of pairs $H_1$ and $H_2$ with
$H_1\le H_2$,
it is convenient to have an invariant
which reflects the relation $H_1\le H_2$.
Ideally,
we want to have an invariant $f$ of graphs
such that $H_1\le H_2$ implies $f(H_1)\le f(H_2)$.
Unfortunately,
the example of $K_{1, k}^+\le K_{1,k}$ suggests that
neither the order nor the size (the number of edges)
satisfies this property.
However,
if we find a constant $c$
such that $H_1\le H_2$ implies
$|V(H_1)|\le |V(H_2)|+c$
or $|E(H_1)|\le |E(H_2)|+c$,
then for a given graph $H_2$,
the number of graphs $H_1\in \HH$ with $H_1\le H_2$ is finite,
which will facilitate our research.
\par
Based on the above motivation,
in this section,
we investigate several basic graph invariants
and whether and how they reflect the relation $\le$.
\par
For a connected graph $G$,
define $\dim G$,
the dimension of $G$,
by $\dim(G)=|E(G)|-|V(G)|+1$.
It is well-known that $\dim G\ge 0$ and
$\dim G = 0$ if and only if $G$ is a tree.
Moreover,
it is easy to see that $H_1\subseteq H_2$ implies
$\dim H_1 \le \dim H_2$.
\par
In Theorems~\ref{invariants}--\ref{dim_monotone},
we prove that the order,
size and dimension serve our purpose.
We also give an inequality on the maximum degree.
\begin{theorem}
Let $H_1, H_2\in\HH$.
If $H_1\le H_2$,
then
\begin{enumerate}
\item
$|E(H_1)|\le |E(H_2)|+1$,
\item
$|V(H_1)|\le |V(H_2)|+2$,
and
\item
$\Delta(H_1)\le \Delta(H_2)+1$.
Moreover,
if $\Delta(H_1)=\Delta(H_2)+1$,
then $H_1$ contains at most two vertices of degree $\Delta(H_1)$,
and if $H_1$ contains two vertices of $\Delta(H_2)+1$,
they are adjacent.
\item
If $H_2$ is not a tree,
then $\dim(H_1)\le\dim(H_2)$.
\end{enumerate}
\label{invariants}
\end{theorem}
\begin{proof}
Since $H_1\le H_2$,
there exists a positive integer $t$
such that every rainbow $H_1$-free edge-colored complete graph colored in $t$ or more colors
is rainbow $H_2$-free.
We may assume $t\ge |E(H_2)|+1$.
Let $|V(H_2)|=n$,
$E(H_2)=\{f_1,\dots, f_m\}$,
$|V(H_1)|=n'$ and
$|E(H_1)|=m'$.
If $H_2$ is not a tree,
we choose $f_1$ so that $H_2-f_1$ is connected.
\par
Introduce $n+2(t-m)$ vertices
$x_1,\dots, x_n, y_{m+1},\dots, y_t, z_{m+1},\dots, z_t$.
Let $X=\{x_1,\dots, x_n\}$,
$Y=\{y_{m+1},\dots, y_t, z_{m+1},\dots, z_t\}$ and $K=K[X\cup Y]$.
Take a graph $G$ with $V(G)=X$ which is isomorphic to $H_2$.
Let $\varphi$ be an isomorphism from $H_2$ to $G$ and
let $e_i=\varphi(f_i)$
($1\le i\le m$).
By the choice of $f_1$,
if $H_2$ is not a tree,
then $G-e_1$ is connected.
\par
Define $c\colon E(K)\to\{1,\dots, t\}$ by
\[
c(e)=
\begin{cases}
i & \text{if $e=e_i$, $1\le i\le m$}\\
j & \text{if $e=y_jz_j$, $m+1\le j\le t$}\\
1 & \text{otherwise.}    
\end{cases}
\]
Then $(K, c)$ is colored in $t$ colors
and $G$ is a rainbow subgraph of $K$ isomorphic to $H_2$.
This implies that $(K, c)$ contains a rainbow subgraph $G'$
which is isomorphic to $H_1$.
Let $F=\bigl(E(K[X])-E(G)\bigr)\cup E_K(X, Y)$.
Note that all the edges in $F$ are colored in $1$.
\par
If $H_1\subseteq H_2$,
then trivially
$|E(H_1)|\le |E(H_2)|$,
$|V(H_1)|\le |V(H_2)|$,
$\Delta(H_1)\le \Delta(H_2)$
and $\dim(H_1)\le \dim(H_2)$.
Thus,
all the statements of the theorem hold.
Therefore,
we may assume $H_1\not\subseteq H_2$.
On the other hand,
a maximal connected rainbow subgraph of $K-X$ is isomorphic to $P_4$.
Since $H_1\in\HH$,
we have
$V(G')\cap X\ne\emptyset$.
Then since $H_1\not\subseteq H_2$ and $G'$ is connected,
we have $E(G')\cap F\ne\emptyset$.
Let $e_0\in E(G')\cap F$.
Since $c(e_0)=1$ and $G'$ is rainbow,
we have $E(G')\cap F=\{e_0\}$,
and $e_1\notin E(G')$.
\par
We consider two cases depending on whether $H_2$ is a tree or not.
\medbreak\noindent
\textbf{Case~1.  $H_2$ is not a tree.}
\medbreak
In this case,
$G-e_1$ is connected.
First,
suppose $e_0\in E(K[X])-E(G)$.
Let $G_1=G-e_1+e_0$.
Then $G_1$ is connected.
In this case,
$E(G')\cap E_K(X, Y)=\emptyset$
and hence $G'$ is a subgraph of $G_1$.
This yields
$|V(G')|\le |V(G_1)|$,
$|E(G')|\le |E(G_1)|$,
$\Delta(G')\le \Delta(G_1)$ and
$\dim G'\le\dim G_1$.
Since $|E(G_1)|=|E(G)|$,
$|V(G_1)|=|V(G)|$
and $\Delta(G_1)\le \Delta(G)+1$,
we have
\begin{align*}
|E(H_1)| &= |E(G')|\le |E(G_1)|=|E(G)|=|E(H_2)|,\\
|V(H_1)| &= |V(G')|\le |V(G_1)|=|V(G)|=|V(H_2)|
\text{ and}\\
\Delta(H_1) &= \Delta(G')\le\Delta(G_1)\le\Delta(G)+1=\Delta(H_2)+1.
\end{align*}
Moreover,
$\deg_{G_1}v=\deg_G v$ for every vertex $v$ in $X=V(G)$
possibly except for the endvertices of $e_0$.
Therefore,
if $\Delta(H_1)=\Delta(H_2)+1$,
then only the vertices in $H_1$ corresponding to
the endvertices of $e_0$ in $\varphi$ attain $\Delta(H_1)$,
and if two vertices in $H_1$ have degree $\Delta(H_2)+1$,
then they are adjacent.
We also have
\[
\begin{split}
\dim H_1 &= \dim G'\le \dim G_1=|E(G_1)|-|V(G_1)|+1\\
&= |E(G)|-|V(G)|+1=\dim G=\dim H_2.
\end{split}    
\]
\par
Next,
suppose $e_0\in E_K(X, Y)$.
In this case,
$V(e_0)\cap Y\ne\emptyset$.
Without loss of generality,
we may assume $e_0=uy_{m+1}$,
where $u\in X$.
Let $G_2=G-e_1+e_0+y_{m+1}z_{m+1}$.
Then $G'$ is a subgraph of $G_2$
and hence $|V(G')|\le |V(G_2)|$,
$|E(G')|\le |E(G_2)|$,
$\Delta(G')\le \Delta(G_2)$ and
$\dim G'\le \dim G_2$.
Note $|E(G_2)|=|E(G)|+1$ and
$|V(G_2)|=|V(G)|+2$.
Moreover,
since $G$ is isomorphic to $H_2$ and $H_2\in\HH$,
we have $\Delta(G)\ge 2$.
Therefore,
$\Delta(G_2)\le\max\{\Delta(G)+1, \deg_{G_2} y_{m+1}\}
=\max\{\Delta(G)+1, 2\}=\Delta(G)+1$.
Then we have
\begin{align*}
|E(H_1)| &= |E(G')|\le |E(G_2)|=|E(G)|+1=|E(H_2)|+1,\\
|V(H_1)| &= |V(G')|\le |V(G_2)| = |V(G)|+2=|V(H_2)|+2,
\text{ and}\\
\Delta(H_1) &= \Delta(G')\le \Delta(G_2)\le \Delta(G)+1=\Delta(H_2)+1.
\end{align*}
Moreover,
if $\Delta(H_1)=\Delta(H_2)+1$,
then since $\Delta(H_1)=\Delta(G)+1\ge 3$,
the same argument as in the previous paragraph
deduces that only $u$ has degree $\Delta(H_1)$.
We also have
\[
\begin{split}
\dim H_1 &= \dim G' \le \dim G_2
=|E(G_2)|-|V(G_2)|+1\\
&=(|E(G)|+1)-(|V(G)|+2)+1=\dim G-1=\dim H_2-1.   
\end{split} 
\]
Therefore,
the theorem follows in this case.
\medbreak\noindent
\textbf{Case~2.  $H_2$ is a tree.}
\medbreak
In this case $G$ is a tree,
and $G-e_1$ has two components $T_1$ and $T_2$,
both of which are trees.
We consider several subcases depending on where $e_0$ lies.
\smallbreak\noindent
\textbf{Subcase~2.1.  $e_0\in E_K(V(T_1), V(T_2))$.}
\smallbreak
Let $G_3=G-e_1+e_0$.
In this case,
$G_3$ is a tree with $V(G_3)=X=V(G)$,
and $G'$ is a subgraph of $G_3$.
Moreover,
$\Delta(G_3)\le\Delta(G)+1$.
Since $|E(G_3)|=|E(G)|$
and $|V(G_3)|=|V(G)|$,
we have
\begin{align*}
|E(H_1)| &= |E(G')|\le |E(G_3)|=|E(G)|=|E(H_2)|,\\
|V(H_1)| &= |V(G')|\le |V(G_3)|=|V(G)|=|V(H_2)|,
\text{ and}\\
\Delta(H_1) &= \Delta(G')\le \Delta(G_3)\le\Delta(G)+1=\Delta(H_2)+1.
\end{align*}
\smallbreak\noindent
\textbf{Subcase~2.2.
$e_0\in \bigl(E(K[V(T_1)])-E(T_1)\bigr)\cup \bigl(E(K[V(T_2)]))-E(T_2)\bigr)$.}
\smallbreak
By symmetry,
we may assume $e_0\in E(K[V(T_1)])-E(T_1)$.
Let $G_4=T_1+e_0$.
Then $G'$ is a subgraph of $G_4$.
Note $|E(G_4)|=|E(G)|-|E(T_2)|$,
$|V(G_4)|=|V(G)|-|V(T_2)|$
and $\Delta(G_4)\le\Delta(G)+1$.
Note that possibly $E(T_2)=\emptyset$,
but $|V(T_2)|\ge 1$.
Therefore,
\begin{align*}
|E(H_1)| &= |E(G')|\le |E(G_4)|=|E(G)|-|E(T_2)|\le |E(G)|=|E(H_2)|,\\
|V(H_1)| &= |V(G')|\le |V(G_4)|=|V(G)|-|V(T_2)|\le |V(G)|-1=|V(H_2)|-1,
\text{ and}\\
\Delta(H_1) &= \Delta(G')\le\Delta(G_4)\le\Delta(G)+1=\Delta(H_2)+1.
\end{align*}
\smallbreak\noindent
\textbf{Subcase~2.3.
$e_0\in E_K(X, Y)$}
\smallbreak
Without loss of generality,
we may assume $e_0=uy_{m+1}$,
where $u\in V(T_1)$.
Let $G_5=T_1+e_0+y_{m+1}z_{m+1}$.
Then $G'\subseteq G_5$.
Since $|E(G_5)|=|E(T_1)|+2=|E(G)|-|E(T_2)|+1\le |E(G)|+1$,
$|V(G_5)|=|V(T_1)|+2=|V(G)|-|V(T_2)|+2\le |V(G)|+1$
and $\Delta(G_5)\le\Delta(G)+1$,
we have
\begin{align*}
|E(H_1)| &= |E(G')|\le |E(G_5)|\le |E(G)|+1=|E(H_2)|+1,\\
|V(H_1)| &= |V(G')|\le |V(G_5)|\le |V(G)|+1=|V(H_2)|+1,
\text{ and}\\
\Delta(H_1) &= \Delta(G')\le\Delta(G_5)\le\Delta(G)+1=\Delta(H)+1.
\end{align*}
Note that in all the subcases,
if $\Delta(H_1)=\Delta(H_2)+1$,
then only the vertices in $H_1$ corresponding to
the endvertices of $e_0$ in $\varphi$ attain $\Delta(H_1)$,
and if two vertices in $H_1$ have degree $\Delta(H_2)+1$,
then they are adjacent.
\end{proof}
\par
The inequalities in (1),
(2) and (3) of Theorem~\ref{invariants}
are all best possible.
See the remark after Theorem~\ref{S221_and_cycles}
\par
Note that in Subcase~2.2,
we have $|E(G_4)|=|E(G)|-|E(T_2)|$
and $|V(G_4)|=|V(G)|-|V(T_2)|$.
These yield
\[
\dim G_4=|E(G_4)|-|V(G_4)|+1
=|E(G)|-|V(G)|+1-(|E(T_2)|-|V(T_2)|)=\dim G +1.    
\]
This equality gives $\dim H_1\le\dim H_2+1$,
but not $\dim H_1\le \dim H_2$.
However,
by a different approach,
we prove $\dim H_1\le \dim H_2$ even if $H_2$ is a tree.
\begin{theorem}\label{dim_tree}
For $H_1\in\HH$ and a tree $T$ in $\HH$,
if $H_1\le T$ holds,
then $H_1$ is a tree.
\end{theorem}
\begin{proof}
Since
$H_1\le T$,
there exists a positive integer $t$
such that every rainbow $H_1$-free edge-colored complete graph
colored in $t$ or more colors is rainbow $T$-free.
We may assume $t\ge |V(T)|$.
\par
Choose one vertex $r$ of $T$ and consider $T$ as a rooted tree
with root $r$.
Then label the vertices of $T$
in the manner of breadth-first search
starting from $r$.
This enables us to assume that
the vertices of $T$ are labeled
as $V(T)=\{x_1,\dots, x_n\}$
so that if $x_i$ is the parent of $x_j$,
then $i < j$ holds.
Note $r=x_1$.
\par
Introduce new vertices $x_0$ and
$x_{n+1},\dots, x_t$,
and let $K=K[\{x_0, x_1,\dots x_t\}]$.
Define $c\colon E(K)\to\{1,\dots, t\}$ by
$c(x_ix_j)=\max\{i, j\}$.
Since $c(x_0x_i)=i$ for $1\le i\le t$,
$K$ is edge-colored in $t$ colors.
\par
Suppose $c(x_{i_1}x_{j_1})=c(x_{i_2}x_{j_2})$
for $\{x_{i_1}x_{j_1}, x_{i_2}x_{i_2}\}\subset E(T)$.
We may assume that $x_{i_1}$ and $x_{i_2}$ are the parents
$x_{j_1}$ and $x_{j_2}$,
respectively.
Then our labeling of $V(T)$ forces
$i_1 < j_1$ and $i_2 < j_2$,
which implies $c(x_{i_1}x_{j_1})=j_1$
and $c(x_{i_2}x_{j_2})=j_2$.
Therefore,
we have $j_1=j_2$.
This implies $i_1=i_2$ and hence
$x_{i_1}x_{j_1}=x_{i_2}x_{j_2}$.
Therefore,
the mapping $c$ is injective
and hence $(T, c)$ is rainbow.
\par
Since $c$ uses $t$ colors and $(K, c)$ contains a rainbow $T$,
$(K, c)$ contains a rainbow subgraph $G'$ which is
isomorphic to $H_1$.
\par 
Assume $G'$ contains a cycle.
Let $C=x_{i_0}x_{i_1}\dots x_{i_p}x_{i_0}$ be a cycle in $G'$.
We may assume $i_0=\max\{i_0, i_1,\dots, i_p\}$.
Then $i_0 > i_1$ and $i_0 > i_p$
and hence $c(x_{i_0}x_{i_1})=c(x_{i_0}x_{i_p})=i_0$.
This contradicts the fact that $(G', c)$ is rainbow.
Therefore,
$G'$ does not contain a cycle.
Since $G'$ is connected,
we see that $G'$ is a tree,
which implies that $H_1$ is also a tree.
\end{proof}
\par
For a later reference,
we combine Theorem~\ref{invariants}~(4)
and~\ref{dim_tree}
into one statement.
\begin{theorem}
For $H_1, H_2\in\HH$,
$H_1\le H_2$ implies $\dim H_1\le\dim H_2$.
\label{dim_monotone}
\end{theorem}
%%%%%%%%%%%%%%%%%%%%%%%%%%%%%%%%%%%%%%%%%%%%%%%%%%%%%%%%%%%%%%
\section{Pairs Comparable in $\boldsymbol{\le}$ but Not in $\boldsymbol{\subseteq}$}
In this section,
we seek pairs of graphs $(H_1, H_2)$ such that $H_1\le H_2$ but
neither $H_1$ nor $H_2$ is a subgraph of the other.
In particular,
we search for examples among trees and cycles.
\subsection{Cycle vs Cycle}
We first search for pairs among cycles.
We remark that no two cycles of different order are comparable with respect to $\subseteq$.
\par
Let $Z_1$ be the graph obtained from $C_3$ by attaching one pendant edge
(Figure~\ref{Z1figure}).
\begin{figure}
\centering
\includegraphics[width=0.15\textwidth]{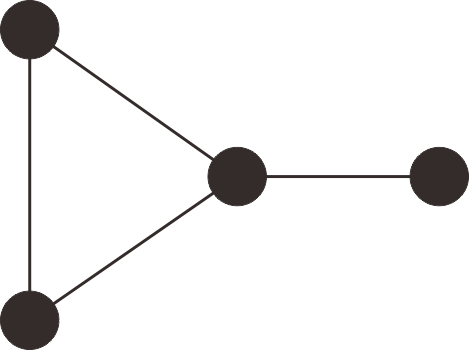}
\caption{$Z_1$}
\label{Z1figure}
\end{figure}
Note that since $C_3\subsetneq Z_1$,
we have $Z_1\not\le C_3$ by Theorem~\ref{clms}.
\begin{theorem}\label{Z1}
For every integer $k$ with $k\ge 4$,
$Z_1\le C_k$.
\end{theorem}
\begin{proof}
We prove that every rainbow $Z_1$-free edge-colored
complete graph $(K, c)$ is rainbow $C_k$-free.
We proceed by induction on $k$.
\par
First,
suppose $k=4$
and assume that $(K, c)$ is not rainbow $C_4$-free.
Let $C=x_1x_2x_3x_4x_1$ be a rainbow cycle of order~$4$ in $(K, c)$.
We may assume $c(x_1x_2)=1$,
$c(x_2x_3)=2$,
$c(x_3x_4)=3$
and $c(x_4x_1)=4$.
Since $K[x_1x_2, x_2x_3, x_1x_3, x_1x_4]$ is isomorphic to $Z_1$
and $(K, c)$ is rainbow $Z_1$-free,
we have $c(x_1x_3)\in\{1, 2, 4\}$.
Similarly,
by considering $K[x_1x_2, x_2x_3, x_1x_3, x_3x_4]$,
we have
$c(x_1x_3)\in\{1, 2, 3\}$.
These imply $c(x_1x_3)\in\{1, 2\}$.
Then either $K[x_1x_3, x_3x_4, x_1x_4, x_1x_2]$
or $K[x_1x_3, x_3x_4, x_1x_4, x_2x_3]$
is a rainbow subgraph of $(K, c)$ isomorphic to $Z_1$.
This is a contradiction,
and the theorem follows in this case.
\par
Next,
suppose $k\ge 5$.
By the induction hypothesis,
$(K, c)$ is rainbow $C_{k-1}$-free.
Assume $(K, c)$ is not rainbow $C_k$-free,
and let $C=x_1x_2\dots x_kx_1$ be a rainbow cycle of order $k$ in $(K, c)$.
We may assume $c(x_ix_{i+1})=i$ for $1\le i\le k-1$
and $c(x_kx_1)=k$.
Since both $K[x_{k-1}x_k, x_kx_1, x_1x_{k-1}, x_1x_2]$
and $K[x_{k-1}x_k, x_kx_1, x_1x_{k-1}, x_{k-1}x_{k-2}]$
are isomorphic to $Z_1$,
we have $c(x_1x_{k-1})\in\{1, k-1, k\}$
and $c(x_1x_{k-1})\in\{k-2, k-1, k\}$.
These imply $c(x_1x_{k-1})\in \{k-1, k\}$.
Then $x_1x_2\dots x_{k-1}x_1$ is a rainbow cycle of order $k-1$.
This is a contradiction,
and the theorem follows.
\end{proof}
\par
Since $C_3\subseteq Z_1$,
we have the following corollary.
\begin{corollary}\label{C3Ck}
For every integer $k$ with $k\ge 3$,
$C_3\le C_k$.
\end{corollary}
By the same proof technique,
we have the following theorem
for cycles of order greater than~$3$.
\begin{theorem}\label{cycles}
For each integer $k$ and $t$ with $k\ge 3$ and $t\ge 1$,
$C_k\le C_{t(k-2)+2}$.
In particular,
for $t\ge 2$,
$C_4\le C_{2t}$.
\end{theorem}
\begin{proof}
We prove that every rainbow $C_k$-free
edge-colored complete graph $(K, c)$ is
rainbow $C_{t(k-2)+2}$-free by induction on $t$.
If $t=1$,
there is nothing to prove.
Suppose $t\ge 2$,
and assume that there is a rainbow cycle $C=x_1x_2,\dots x_{t(k-2)+2}x_1$
in a rainbow $C_k$-free edge-colored complete graph $(K, c)$.
Without loss of generality,
we may assume $c(x_{i-1}x_i)=i$ for
$1\le i\le t(k-2)+2$,
where we consider $x_{t(k-2)+3}=x_1$.
Let $D=x_1x_2\dots x_kx_1$
and $D'=x_kx_{k+1}x_{k+2}\dots x_{t(k-2)+2}x_1 x_k$.
Then $D'$ is a cycle of order $(t-1)(k-2)+2$.
By the induction hypothesis,
$(K, c)$ is rainbow $C_{(t-1)(k-2)+2}$-free,
and hence $D'$ is not a rainbow cycle.
Since $D'-x_kx_1$ is a rainbow path,
this implies $c(x_1x_k)\in [k, t(k-2)+2]$.
Then $D$ is a rainbow cycle of order $k$
edge-colored in $\{1, 2,\dots, k-1, c(x_1x_k)\}$.
This is a contradiction.
\end{proof}
When we look at Corollary~\ref{C3Ck} and Theorem~\ref{cycles},
we may hope that $C_4\le C_k$ holds for sufficiently large odd integer $k$.
Unfortunately,
it is false.
\begin{theorem}
For every odd integer $k$,
$C_4\not\le C_k$.
\label{C4Ck}
\end{theorem}
\begin{proof}
Assume,
to the contrary,
that $C_4\le C_k$ for some odd integer $k$.
Then there exists a positive integer $t$ such that
every rainbow $C_4$-free edge-colored complete
graph colored in $t$ or more colors
is rainbow $C_k$-free.
We may assume $t\ge k$.
\par
Introduce $k+2(t-k)$ vertices
$x_1, x_2,\dots, x_k, y_{k+1}, \dots y_t, z_{k+1},\dots, z_t$,
and let $X=\{x_1, x_2,\dots, x_k\}$
and $Y=\{y_{k+1}, y_{k+2},\dots, y_t, z_{k+1}, z_{k+2},\dots, z_t\}$.
Also let $K=K[X\cup Y]$ and define
$c\colon E(K)\to\naturalnumbers$
by
\[
c(e) =
\begin{cases}
1 & \text{if $e=x_ix_j$, $1\le i < j \le k$, and $i\equiv j\pmod{2}$}\\
\max\{i, j\}
& \text{if $e=x_ix_j$, $1\le i < j \le k$, and $i\not\equiv j\pmod{2}$}\\
j & \text{if $e=y_jz_j$, $k+1\le j\le t$}\\
1 & \text{otherwise.}
\end{cases}
\]
Note $c(x_ix_j)\ge 2$ if $i\not\equiv j\pmod{2}$
by the definition of $c$.
\par
Let $C = x_1x_2x_3\dots x_kx_1$.
Then $c(x_{i-1}x_i)=i$
($2\le i\le k$).
Moreover,
$c(x_kx_1)=1$ since $k$ is odd.
Therefore,
$C$ is a rainbow cycle of order $k$.
\par
We claim that $(K, c)$ is rainbow $C_4$-free.
Assume,
to the contrary,
that $(K, c)$ contains a rainbow cycle $D$ of order~$4$.
Since $K[Y]$ does not contain a rainbow cycle
and all the edges joining $X$ and $Y$ are colored in~$1$,
$D$ is contained in $K[X]$.
Let $D=x_{i_1}x_{i_2}x_{i_3}x_{i_4}x_{i_1}$.
Without loss of generality,
we may assume that $i_1$ is the largest index
among $\{i_1, i_2, i_3, i_4\}$.
Then since $D$ is rainbow,
$\{c(x_{i_1}x_{i_2}), c(x_{i_1}x_{i_4})\}=\{1, i_1\}$.
By symmetry,
we may assume $c(x_{i_1}x_{i_2})=1$ and
$c(x_{i_1}x_{i_4})=i_1$.
Then by the definition of $c$,
we have $i_1\equiv i_2\pmod{2}$
and $i_1\not\equiv i_4\pmod{2}$.
These imply $i_2\not\equiv i_4\pmod{2}$.
Then $i_3\equiv i_2\pmod{2}$ or $i_3\equiv i_4\pmod{2}$ and hence
$c(x_{i_2}x_{i_3})=1$ or $c(x_{i_3}x_{i_4})=1$.
This contradicts the fact that $D$ is rainbow.
\end{proof}
\subsection{Tree vs Cycle}
Next,
we investigate the relations between a tree and a cycle.
Given a tree $T$ and a cycle $C$,
we see by Theorem~\ref{dim_monotone}
that $C\le T$ never happens.
Therefore,
we have only to study the possibilities of $T\le C$.
\par
We call a vertex of degree at least $3$ a branch vertex.
By Theorem~\ref{invariants},
if $T\le C$,
then $\Delta(T)\le 3$ and $T$ has at most two branch vertices.
Moreover,
it $T$ has two branch vertices,
then they are adjacent.
\par
We first deal with the case in which $T$ has no branch vertex,
i.e.~$T$ is a path.
Note that
in the domain of $\HH$,
$P_5$ is the smallest path.
Let $k$ and $l$ be integers with $k\ge 3$ and $l\ge 5$.
If $l\ge k$,
then $P_k\subseteq C_l$ and hence $P_k\le C_l$ holds.
On the other hand,
by Theorem~\ref{invariants}~(1),
$P_k\not\le C_l$ if $l\le k-3$.
Therefore,
we only consider the case $k-2\le l\le k-1$.
For $l=k-2$,
we settle the problem as follows.
\begin{theorem}
For $k\ge 5$,
$P_k\not\le C_{k-2}$.
\label{Pk_Ck-2}
\end{theorem}
\begin{proof}
Assume $P_k \le C_{k-2}$,
Then there exists a positive integer $t$ such that
every rainbow $P_k$-free edge-colored complete graph $(K, c)$ colored
in $t$ or more colors is rainbow $C_{k-2}$-free.
We may assume $t\ge k-2$.
\par
Take $t$ vertices $x_0, x_1,\dots, x_{k-4} x_{k-3},\dots, x_{t-1}$.
Let $X=\{x_0, x_1,\dots, x_{t-1}\}$ and $K=K[X]$.
Define $c\colon E(K)\to\naturalnumbers$ by
\[
c(e) =
\begin{cases}
i & \text{if $e=x_{i-1}x_i, \quad 1\le i \le k-4$}\\
j & \text{if $e=x_0x_j, \quad k-3 \le j \le t-1$}\\
t & \text{otherwise.}    
\end{cases}
\]
Then $(K, c)$ is colored in $t$ colors and $x_0x_1\dots,x_{k-4}x_{k-3},x_{k-3}x_0$ is a rainbow
cycle of order $k-2$ colored in $\{1,\dots, k-4, t, k-3\}$.
On the other hand,
every rainbow path in $(K, c)$ can contain at most $2$ edges incident with $x_0$ and
the colors that appear in $K-x_0$ are $2,\dots, k-4, t$.
Therefore,
every rainbow path in $(K, c)$ can contain at most $k-2$ edges,
and hence $(K, c)$ is rainbow $P_k$-free.
This is a contradiction.
\end{proof}
\par
For $l=k-1$,
we do not know the answer in general.
However,
we know that $P_k\le C_{k-1}$ holds for $k\in\{5, 6\}$.
The fact $P_5\le C_4$ follows from the result in the next subsection
(Corollary~\ref{P5C4}).
We prove $P_6\le C_5$ here.
\begin{theorem}
$P_6\le C_5$
\label{P6C5}
\end{theorem}
\begin{proof}
We prove that every rainbow $P_6$-free edge-colored complete graph
colored in $14$ or more colors is rainbow $C_5$-free.
\par
Assume,
to the contrary,
that there exists
a rainbow $P_6$-free edge-colored complete graph
$(K, c)$ colored in $t$ colors,
where $t\ge 14$,
which contains a rainbow cycle $C=x_1x_2x_3x_4x_5x_1$.
We may assume $c(E(K))=[1, t]$ and
$c(x_ix_{i+1})=i$ for $1\le i\le 5$,
where the indices of $x_i$ are taken modulo~$5$
in the range $[1, 5]$.
Since $K[V(C)]$ contains $10$ edges,
we may assume $c(E(K[V(C)]))\subset [1, 10]$.
Then since $t\ge 14$,
$V(K)-V(C)\ne\emptyset$.
For $x_i\in V(C)$ and a vertex $v$ in $V(K)-V(C)$,
consider the path $vx_ix_{i+1}x_{i+2}x_{i+3}x_{i+4}$.
Since $k\ge 5$,
it is a path of order~$6$ colored in
$\{c(x_iv), i, i+1, i+2, i+3\}$.
Since $(K, c)$ is rainbow $P_6$-free,
we have $c(x_iv)\in\{i, i+1, i+2, i+3\}$.
This implies that
every edge in $E_K(V(C), V(K)-V(C))$ receives a color in $[1, 5]$,
and hence the edges that receive a color in $[11, t]$ lie in $K-V(C)$.
\par
Let $y\in V(K)-V(C)$.
We claim that
if there exists an edge which is incident with $y$ and receives a color in $[11, t]$,
then $c(x_iy)=i+2$ holds for every $i$ with $1\le i\le 5$.
Let $e=yz$ be an edge in $K-V(C)$ with c$(e)\in[11, t]$.
Then $z\in V(K)-V(C)$.
Let $Q_1=zyx_ix_{i+1}x_{i+2}x_{i+3}$.
Then $Q_1$ is a path of order~$6$ colored in
$\{c(e), c(x_iy), i, i+1, i+2\}$.
Since $(K, c)$ is rainbow $P_6$-free while $c(e)\notin [1, 10]$ and
$c(x_iy)\in [1, 5]$,
we have $c(x_iy)\in\{i, i+1, i+2\}$.
By applying the same argument to
$zyx_ix_{i-1}x_{i-2}x_{i-3}$,
we also have $c(x_iy)\in\{i-1, i-2, i-3\}$.
Therefore,
noting $i+2\equiv i-3\pmod{5}$,
we have $c(x_iy)=i+2$,
and the claim follows.
\par
Re-color all the edges in $K-V(C)$ that receive a color in $[1, 10]$
in a new color $t+1$.
Let $c'$ be the resulting new coloring of $K-V(C)$.
Since $t\ge 14$,
$(K-V(C), c')$ is colored in at least~$4$ colors.
Then by Theorem~\ref{tw},
$(K-V(C), c')$ contains a rainbow path of order~$4$,
which implies that
$(K-V(C), c)$ contains a rainbow path $P=y_1y_2y_3y_4$
such that at most one edge in $P$ receives a color in $[1, 10]$.
Let $c(y_1y_2)=a_1$,
$c(y_2y_3)=a_2$ and $c(y_3y_4)=a_3$.
By the symmetry of the edges $y_1y_2$ and $y_3y_4$,
we may assume $a_1\notin [1, 10]$.
\par
Let $Q_2 = y_1y_2y_3x_1x_2x_3$.
Since at least one edge in $\{y_2y_3, y_3y_4\}$ receives a color
in $[11, t]$,
we have $c(x_1y_3)=3$ by the above claim.
Hence $Q_2$ is a path of order~$6$
colored in $\{a_1, c(y_2y_3), 3, 1, 2\}$.
Since $(K, c)$ is rainbow $P_6$-free,
$a_1\notin [1, 10]$
and $P$ is a rainbow path,
we have $c(y_2y_3)\in\{1, 2, 3\}$.
By applying the same argument to the path
$y_1y_2y_3x_1x_5x_4$,
we also have $c(y_2y_3)\in\{3, 4, 5\}$.
These imply $c(y_2y_3)=3$.
\par
Now let $Q_3=y_1y_2y_3x_4x_5x_1$.
Since $c(x_4y_3)=1$ by the claim,
$Q_3$ is a path of order~$6$,
edge-colored in $\{1, 3, 4, 5, a_1\}$.
This is a rainbow path,
which contradicts the hypothesis.
\end{proof}
\par
Next,
we consider trees with one branch vertex.
We call a tree with exactly one branch vertex a \textsl{spider}.
A path in a spider $S$ starting at its branch vertex
and ending at an endvertex is called a \textit{leg} of $S$,
and we call a spider having $k$ legs a \textit{$k$-legged spider}.
For positive integers $a$,~$b$ and~$c$,
We denote by $S_{a, b, c}$ the $3$-legged spider with legs
of length $a$,
$b$ and $c$.
Note $|V(S_{a, b, c})|=a+b+c+1$.
\par
The smallest $3$-legged spider is $S_{1, 1, 1}=K_{1,3}$,
and the second smallest one is $S_{2, 1, 1}=K^+_{1,3}$.
These two spiders play a special role in the preordered set $(\HH, \le)$
and we will discuss them in the next section.
\par
There are two $3$-legged spiders of order $6$~$\colon$~$S_{2, 2, 1}$
and~$S_{3, 1, 1}$
(Figure~\ref{spiders_figure}).
\begin{figure}
\centering
\includegraphics[width=0.4\textwidth]{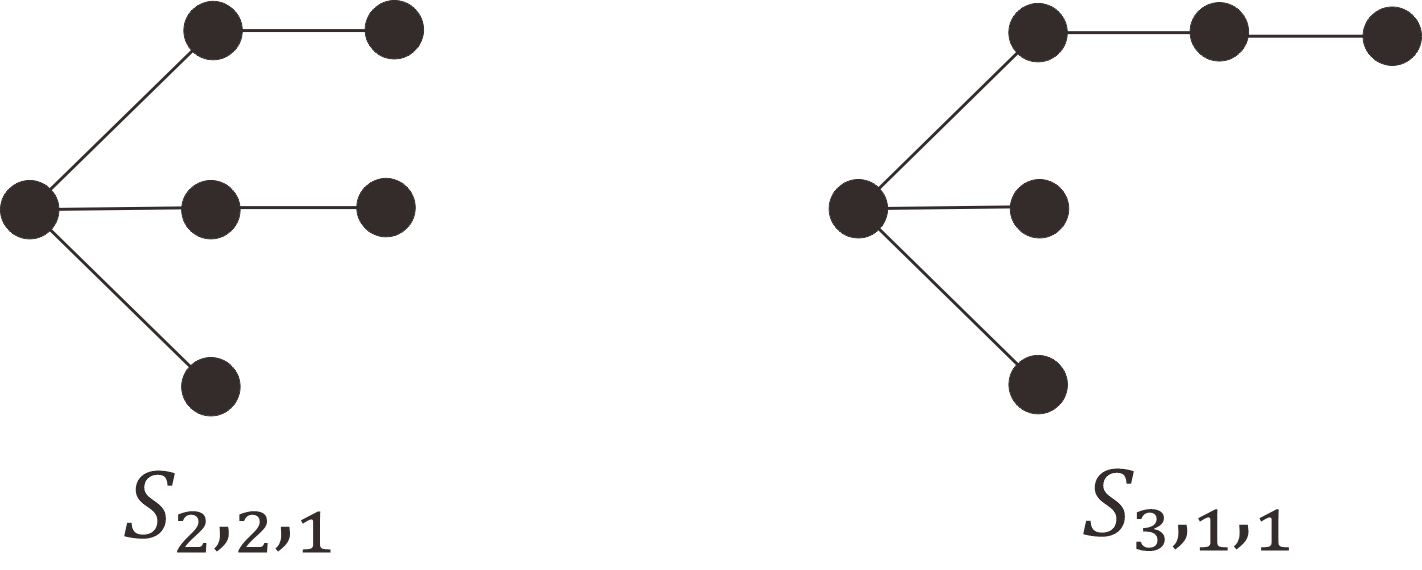}
\caption{$S_{2,2,1}$ and $S_{3,1,1}$}
\label{spiders_figure}
\end{figure}
We will determine which cycles $C$ satisfy $S_{2,2,1}\le C$ and
$S_{3, 1, 1}\le C$.
Note $S_{2,2,1}\not\le C_3$ and $S_{3,1,1}\not\le C_3$ by Theorem~\ref{invariants}~(1).
\begin{theorem}
$S_{2, 2, 1}\le C_k$ for every $k\ge 4$.
\label{S221_and_cycles}
\end{theorem}
Note $|V(S_{2, 2, 1})|=|V(C_4)|+2$,
$|E(S_{2, 2, 1})|=|E(C_4)|+1$ and
$\Delta(S_{2, 2, 1})=\Delta(C_4)+1$,
which attain the equality
in the inequalities of
Theorem~\ref{invariants}~(1),~(2)
and~(3).
\begin{proof}[\textbf{Proof of Theorem~\ref{S221_and_cycles}}]
We prove that every rainbow $S_{2, 2, 1}$-free edge-colored complete graph
$(K, c)$ colored in $\frac{1}{2}k(k-1)+2k+4$ or more colors
is rainbow $C_k$-free.
\par
Assume,
to the contrary,
that there exists a rainbow $S_{2, 2, 1}$-free edge-colored complete graph $(K, c)$
colored in $t$ colors,
where $t \ge \frac{1}{2}k(k-1)+2k+4$,
which contains a rainbow cycle $C=x_1x_2\dots x_kx_1$.
We may assume $c(x_ix_{i+1})=i$
$(1\le i\le k)$,
where the indices are taken modulo~$k$
in the range of $[1, k]$.
Since $|E(K[V(C)])|=\frac{1}{2}k(k-1)$,
we may assume $c(E(K[V(C)]))\subset \left[1, \frac{1}{2}k(k-1)\right]$.
\par
We claim that
for every $x_i\in V(C)$,
at most $2$ colors
in $\left[\frac{1}{2}k(k-1)+1, t\right]$
appear in the edges joining $x_i$ and $V(K)-V(C)$.
Assume,
to the contrary,
that there exist three edges $x_iy_1$,
$x_iy_2$ and $x_iy_3$ with $\{y_1, y_2, y_3\}\subset V(K)-V(C)$
such that $x_iy_j$ is colored in $a_j$
($j=1, 2, 3$),
where $a_1$,
$a_2$ and $a_3$ are three distinct colors in
$\left[\frac{1}{2}k(k-1)+1, t\right]$.
Let $T_1=K[x_iy_1, y_1y_2, x_ix_{i+1}, x_{i+1}x_{i+2}, x_ix_{i-1}]$.
Since $k\ge 4$,
$x_{i+2}\ne x_{i-1}$ and hence $T_1$ is isomorphic to $S_{2,2,1}$.
Then since $T_1$ is edge-colored in $\{i-1, i, i+1, a_1, c(y_1y_2)\}$ while
$(K, c)$ is rainbow $S_{2,2,1}$-free,
we have $c(y_1y_2)\in \{i-1, i, i+1, a_1\}$.
Next,
let $T_2=K[x_iy_2, y_2y_1, x_ix_{i-1}, x_{i-1}x_{i-2}, x_iy_3]$.
Then $T_2$ is isomorphic to $S_{2,2,1}$ and edge-colored
in $\{i-1, i-2, a_2, a_3, c(y_1y_2)\}$,
which yields $c(y_1y_2)\in\{i-1, i-2, a_2, a_3\}$.
These imply $c(y_1y_2)=i-1$.
\par
Now consider $T_3=K[x_iy_1, y_1y_2, x_ix_{i+1}, x_{i+1}x_{i+2}, x_iy_3]$.
Then $T_3$ is a rainbow subgraph of $(K, c)$ isomorphic to $S_{2,2,1}$
edge-colored in $\{i-1, i, i+1, a_1, a_3\}$.
This is a contradiction,
and the claim follows.
\par
By this claim,
we may assume that the colors that appear in the edges
belonging to $E(K[V(C)])\cup E_K(V(C), V(K)-V(C))$
are in $[1, \frac{1}{2}k(k-1)+2k]$.
This implies that all the edges colored in $\left[\frac{1}{2}k(k-1)+2k+1, t\right]$
appear in $K-V(C)$.
\par
Take every edge in $K-V(C)$ receiving a color in $\left[1, \frac{1}{2}k(k-1)+2k\right]$
and re-color it
in a new color $t+1$.
Let $c'$ be the resulting coloring of $K-V(C)$.
Note that possibly no edge in $(K-V(C), c)$ receives a color in $\left[1, \frac{1}{2}k(k-1)+2k\right]$
and hence the color $t+1$ may not arise in $K-V(C)$.
Since $t\ge \frac{1}{2}k(k-1)+2k+4$,
$(K-V(C), c')$ is colored in at least $4$ colors.
Then by Theorem~\ref{tw},
$(G-V(C), c')$ contains a rainbow path of order~$4$, 
which implies that $(K-V(C), c)$ contains a rainbow path of order~$4$
containing at most one edge colored in
$\left[1, \frac{1}{2}k(k-1)+2k\right]$.
\par
Let $P=z_1z_2z_3z_4$ be a rainbow path of order $4$
in $G-V(C)$ in which at most one edge receives a color
in $\left[\frac{1}{2}k(k-1)+2k\right]$.
Let $c(z_1z_2)=b_1$,
$c(z_2z_3)=b_2$ and $c(z_3z_4)=b_3$.
By the symmetry of the edges $z_1z_2$ and $z_3z_4$,
we may assume $b_1\notin \left[1, \frac{1}{2}k(k-1)+2k\right]$.
\par
Now,
we take $x_i\in V(C)$ and determine possible colors of $c(x_iz_2)$.
Let $T_4=K[x_iz_2, z_2z_1, x_ix_{i+1},$
$x_{i+1}x_{i+2}, x_ix_{i-1}]$.
Since $k\ge 4$,
we have $x_{i-1}\ne x_{i+2}$ and hence
$T_4$ is isomorphic to $S_{2, 2, 1}$.
Since $T_4$ is edge-colored in $\{i-1, i, i+1, b_1, c(x_iz_2)\}$
while $c(x_1z_2)\in \left[1, \frac{1}{2}k(k-1)+2k\right]$
and $b_1\notin\left[1, \frac{1}{2}k(k-1)+2k\right]$,
we have $c(x_iz_2)\in\{i-1, i, i+1\}$.
Similarly,
by considering $K[x_iz_2, z_2z_1, x_ix_{i-1},$
$x_{i-1}x_{i-2}, x_ix_{i+1}]$,
we also have $c(x_iz_2)\in \{i-2, i-1, i\}$.
Therefore,
we have
$c(x_iz_2)\in \{i-1, i\}$
for each $i$ with $1\le i\le k$.
\par
If $\{b_2, b_3\}\cap[1, k]=\emptyset$,
Then let
\[
T_5 =
\begin{cases}
K[z_2z_3, z_3z_4, z_2x_2, x_2x_3, z_2z_1]
& \text{if $c(x_2z_2)=1$}\\
K[z_2z_3, z_3z_4, z_2x_2, x_2x_1, z_2z_1]
& \text{if $c(x_2z_2)=2$.}    
\end{cases}
\]
Then in either case,
$T_5$ is a rainbow subgraph isomorphic to $S_{2, 2, 1}$ and edge-colored
in $\{1, 2, b_1, b_2, b_3\}$.
This is a contradiction.
Thus,
we have $b_2\in[1, k]$ or $b_3\in [1, k]$.
Let $\{b_2, b_3\}=\{i, \beta\}$,
where $i\in [1, k]$.
Then $\beta\notin \left[1, \frac{1}{2}k(k-1)+2k\right]$.
\par
Let
\[
T_6 = 
\begin{cases}
K[z_2z_3, z_3z_4, z_2x_{i+2}, x_{i+2}x_{i+3}, z_2z_1]
& \text{if $c(x_{i+2}z_2)=i+1$}\\
K[z_2z_3, z_3z_4, z_2x_{i+2}, x_{i+2}x_{i+1}, z_2z_1]
& \text{if $c(x_{i+2}z_2)=i+2$.}
\end{cases}
\]
Then $T_5$ is a rainbow subgraph of $(K, c)$
isomorphic to $S_{2,2,1}$ and edge-colored
in $\{b_1, \beta, i, i+1, i+2\}$.
This is a final contradiction,
and the theorem follows.
\end{proof}
\par
Since $P_5\subset S_{2, 2, 1}$,
we have the following corollary.
\begin{corollary}
$P_5\le C_4$
\label{P5C4}
\end{corollary}
\begin{theorem}
\leavevmode \vspace{-1ex}
\begin{enumerate}
\item
$S_{3, 1, 1}\not\le C_4$
\item
For each $k\ge 5$,
$S_{3, 1, 1}\le C_k$.
\end{enumerate}
\label{S311_and_cycles}
\end{theorem}
\begin{proof}
(1)\ Assume $S_{3, 1, 1}\le C_4$.
Then there exists a positive integer $t$
such that every rainbow $S_{3, 1, 1}$-free edge-colored complete graph
colored in $t$ or more colors is
rainbow $C_4$-free.
We may assume $t\ge 4$.
\par
Introduce $4+2(t-4)$ vertices
$x_1, x_2, x_3, x_4$,
$y_5,\dots, y_t$
and $z_5\dots, z_t$,
and let $X=\{x_1, x_2, x_3, x_4\}$
and $Y=\{y_5,\dots, y_t, z_5,\dots, z_t\}$.
Also,
let $K=K[X\cup Y]$.
Define $c\colon E(K)\to\naturalnumbers$ by
\[
c(e) = 
\begin{cases}
i & \text{if $e=x_ix_{i+1}$, $1\le i\le 4$}\\
j & \text{if $e=y_jz_j$, $5\le j\le t$}\\
1 & \text{otherwise,}
\end{cases}
\]
where we consider $x_5=x_1$.
Then $(K, c)$ is colored in $t$ colors.
Moreover,
$x_1x_2x_3x_4x_1$ is a rainbow cycle of order~$4$
edge-colored in $\{1, 2, 3, 4\}$.
Thus,
$(K, c)$ contains a rainbow subgraph $T$
which is isomorphic to $S_{3,1,1}$.
\par
Let $u$ be the branch vertex of $T$.
By the definition of $c$,
$\Delta_c(K, c)=3$ and
$x_3$ and $x_4$ are the only vertices of color degree~$3$
in $(K, c)$.
Hence we have either $u=x_3$ or $u=x_4$.
By symmetry,
we may assume $u=x_3$.
Then $x_3x_2, x_3x_4\in E(T)$
and the edge in $T$ incident with $x_3$ other than $x_3x_2$ and $x_3x_4$
is colored in~$1$.
This implies $x_1x_2\notin E(T)$
and hence every leg of $T$ has length at most~$2$.
This contradicts the assumption that $T$ is isomorphic to $S_{3,1,1}$.
\medbreak\noindent
(2)~We prove that every rainbow $S_{3, 1, 1}$-free edge-colored complete graph
$(K, c)$ colored in $\frac{1}{2}k(k-1)+4$ or more colors is rainbow $C_k$-free.
Assume,
to the contrary,
that there exists a rainbow $S_{3, 1, 1}$-free
edge-colored complete graph $(K, c)$
colored in $t$ colors,
where $t\ge\frac{1}{2}k(k-1)+4$,
which contains a rainbow cycle
$C=x_1x_2\dots x_kx_1$.
We may assume $c(x_ix_{i+1})=i$,
$1\le i\le k$,
where the indices of $x_i$ are taken modulo~$k$
in the range of $[1, k]$.
Since $|E(K[V(C)])|=\frac{1}{2}k(k-1)$,
we may assume $c(E(K[V(C)]))\subset\left[1, \frac{1}{2}k(k-1)\right]$.
Since $t\ge\frac{1}{2}k(k-1)+4$,
this implies $V(K)-V(C)\ne\emptyset$.
\par
Take an edge $x_iy$,
with $1\le i\le k$ and $y\in V(K)-V(C)$.
let $T_1=K[x_ix_{i+1}, x_{i+1}x_{i+2},$
$x_{i+2}x_{i+3}, x_ix_{i-1}, x_iy]$.
Since $k\ge 5$,
$x_{i+3}\ne x_{i-1}$ and hence
$T_1$ is isomorphic to $S_{3, 1, 1}$.
Also $T_1$ is edge-colored in $\{i-1, i, i+1, i+2, c(x_iy)\}$.
Since $(K, c)$ is rainbow $S_{3, 1, 1}$-free,
we have $c(x_iy)\in\{i-1, i, i+1, i+2\}$.
This implies that every edge joining $V(C)$ and $V(K)-V(C)$ is assigned a color in $[1, k]$.
Therefore,
all the edges colored in $\left[\frac{1}{2}k(k-1)+1, t\right]$ lie in $K-V(C)$.
\par
Take every edge in $K-V(C)$ receiving a color in $\left[1, \frac{1}{2}k(k-1)\right]$
and re-color it in a new color $t+1$.
Let $c'$ be the resulting coloring of $K-V(C)$.
Since $t\ge \frac{1}{2}k(k-1)+4$,
$(K-(C), c')$ is colored in at least $4$ colors.
Then by Theorem~\ref{tw},
$(K-V(C), c')$ contains a rainbow path of order~$4$,
which implies that $(K-V(C), c)$ contains a rainbow path of order~$4$
containing at most one edge colored in
$\left[1, \frac{1}{2}k(k-1)\right]$.
\par
Let $P=z_1z_2z_3z_4$ be a rainbow path of order $4$ in ($K-V(C), c)$
in which at most one edge receives a color in $\left[1, \frac{1}{2}k(k-1)\right]$.
Let $c(z_1z_2)=b_1$,
$c(z_2z_3)=b_2$ and $c(x_3x_4)=b_3$.
By the symmetry of the edges $z_1z_2$ and $z_3z_4$,
we may assume $b_1\notin \left[1, \frac{1}{2}k(k-1)\right]$.
\par
For $i\in [1, k]$,
let $T_2=K[z_2x_i, x_ix_{i+1}, x_{i+1}x_{i+2}, z_2z_1, z_2z_3]$.
Then $T_2$ is isomorphic to $K_{3,1,1}$ and
edge-colored in $\{i, i+1, b_1, b_2, c(x_iz_2)\}$.
Since $(K, c)$ is rainbow $S_{3,1,1}$-free while
$c(x_iz_2)\in[1, k]$ and $b_1\notin \left[1, \frac{1}{2}k(k-1)\right]$,
we have $c(x_iz_2)\in \{i, i+1, b_2\}$.
By applying the same argument to
$K[z_2x_i, x_ix_{i-1}, x_{i-1}x_{i-2}, z_2z_1, z_2z_3]$,
we have $c(x_iz_2)\in \{i-1, i-2, b_2\}$.
Since $k\ge 5$,
we have
$\{i, i+1\}\cap\{i-1, i-2\}=\emptyset$ and hence
$c(x_iz_2)=b_2$ for every $i$ with $1\le i\le k$.
In particular,
$b_3 \notin\left[1, \frac{1}{2}k(k-1)\right]$,
and since every edge in $E_K(V(C), V(K-V(C)))$
is colored in $[1, k]$,
we have 
$b_2\in [1,k]$.
\par
Let $T_3=K[x_{b_2+2}x_{b_2+3}, x_{b_2+3}x_{b_2+4}, x_{b_2+4}x_{b_2+5},
x_{b_2+2}x_{b_2+1}, x_{b_2+2}z_2]$.
Since $k\ge 5$,
$x_{b_2+1}\ne x_{b_2+5}$ and hence
$T_3$ is isomorphic to $S_{3,1,1}$.
Moreover,
since $c(x_{b_2+2}z_2)=b_2$,
$T$ is a rainbow subgraph
edge-colored in $\{b_2, b_2+1, b_2+2, b_2+3, b_2+4\}$.
This is a contradiction,
and the theorem follows.
\end{proof}
\par
Finally,
we consider trees containing two branch vertices of degree~$3$.
For these trees,
we know little.
We only determine the relation between the smallest one,
which is a barbell,
and cycles.
\begin{theorem}
\leavevmode
\begin{enumerate}
\item
$B\le C_k$ for each $k\ge 6$.
\item 
$B\not\le C_k$ for $k\in\{3, 4, 5\}$.
\end{enumerate}
\label{barbell_cycles}
\end{theorem} 
\begin{proof}
(1)\ \ {}
We prove that for every integer $t$ with $t\ge\frac{1}{2}k(k+1)+4$,
every rainbow $B$-free edge-colored complete graph
$(K, c)$ colored in $t$ colors is rainbow $C_k$-free.
Assume,
to the contrary,
that there exist an integer $t$ with $t\ge \frac{k(k+1)}{2}+4$
and a $B$-free edge-colored complete graph $(K, c)$ colored in $t$ colors
such that $(K, c)$ contains a rainbow cycle $C=x_1x_2\dots x_kx_1$ of order $k$.
We may assume that $c\colon E(G)\to [1, t]$.
We may also assume that $c(x_ix_{i+1})=i$
for $1\le i\le k$,
where indices are taken modulo $k$ in $[1, k]$.
Also,
since $|E(K[V(C)])|=\frac{1}{2}k(k-1)$,
we may assume $c(V(K[V(C)]))\subset \left[1, \frac{1}{2}k(k-1)\right]$.
\par
For indices $i$ and $j$,
we call the pair $(i, j)$
a \textit{distant pair\/}
if both $l(x_i\Cforw x_j)\ge 3$ and
$l(x_j\Cforw x_i)\ge 3$ hold.
Also,
for a distant pair $(i, j)$,
we define $I(i, j)$ by
$I(i, j)=\{i-1, i, j-1, j\}$.
Note that an index $i$ also represents a color of the edge $x_ix_{i+1}$.
We remark the following.
\begin{enumerate}
\item[(i)]
If $(i, j)$ is a distant pair,
then $K[x_ix_{i-1}, x_ix_{i+1}, x_ix_j, x_jx_{j-1}, x_jx_{j+1}]$
is a barbell edge-colored in $I(i, j)\cup\{c(x_ix_j)\}$.
Since $(K, c)$ is rainbow $B$-free,
this implies $c(x_ix_j)\in I(i, j)$.
Then we can define $p(i, j)$ and $q(i, j)$
by $\{p(i, j), q(i, j)\}=\{i, j\}$,
$c(x_ix_j)\in \{p(i, j)-1, p(i, j)\}$
and $c(x_ix_j)\notin \{q(i, j)-1, q(i, j)\}$.
\item[(ii)]
A distant pair is defined based on the distance between two vertices in $C$.
Therefore,
if two vertices in $C$ form a distant pair for some labeling of $x_i$,
then they form a distant pair even if we reverse the order of indices.
\end{enumerate}
\begin{claim}\label{few_edges}
For each $i$ with $1\le i\le k$,
at most one color in $\left[\frac{1}{2}k(k-1)+1, t\right]$
appears in the edges joining $x_i$ and $V(K)-V(C)$.
\end{claim}
\claimproof\ \ {}
Assume,
to the contrary,
there exist two edges $x_iu$ and $x_iv$ with $\{u, v\}\subset V(K)-V(C)$
with $\{c(x_iu), c(x_iv)\}\subset \left[\frac{1}{2}k(k-1)+1, t\right]$
and $c(x_iu)\ne c(x_iv)$.
Let $c(x_iu)=c_u$ and $c(x_iv)=c_v$.
Since $K[x_iu, x_iv, x_ix_{i+1} x_{i+1}x_{i+2}, x_{i+1}x_{i-1}]$
is a barbell edge-colored in $\{c_u, c_v, i, i+1, c(x_{i+1}x_{i-1})\}$ and
$c(x_{i-1}x_{i+1})\in \left[1, \frac{1}{2}k(k-1)\right]$,
we have $c(x_{i-1}x_{i+1})\in \{i, i+1\}$.
However,
$K[x_iu, x_iv, x_ix_{i-1}, x_{i-1}x_{i-2}, x_{i-1}x_{i+1}]$
is also a barbell and it is edge-colored in
$\{c_u, c_v, i-1, i-2, c(x_{i-1}x_{i+1})\}$.
Since $c(x_{i-1}x_{i+1})\in\{i, i+1\}$
and $k\ge 6$,
this subgraph is rainbow,
which contradicts the assumption that $(K, c)$ is rainbow $B$-free.
\qed
\par
By Claim~\ref{few_edges},
The edges in $E(K[V(C)])\cup E_K(V(C), V(K)-V(C))$
are colored in at most $\frac{1}{2}k(k-1)+k=\frac{1}{2}k(k+1)$ colors.
Without loss of generality,
we may assume $c(E(K[V(C)])\cup E_K(V(C), V(K)-V(C)))\subset\left[1, \frac{1}{2}k(k+1)\right]$.
Since $K$ is edge-colored in $\frac{k(k+1)}{2}+t$ colors,
we see that colors in $\left[\frac{1}{2}k(k+1)+1, \frac{1}{2}k(k+1)+t\right]$
only appear in the edges of $K[V(K)-V(C)]$.
\par
Now re-color the edges in $K[V(K)-V(C)]$ receiving a color in
$\left[1, \frac{1}{2}k(k+1)\right]$ in a new color $t+1$.
Let $c'$ be the resulting color of $K-V(C)$.
Then whether the color $t+1$ arises in $c'$ or not,
$(K-V(C), c')$ uses at least $t\ge 4$ colors,
and hence it contains a rainbow path of order~$4$
by Theorem~\ref{tw}.
This implies that in $(K, c)$,
$K-V(C)$ contains a rainbow path $P=z_1z_2z_3z_4$
which contains at most one edge
having a color in $\left[1, \frac{1}{2}k(k+1)\right]$.
Let $c(z_iz_{i+1})=b_i$ for $i\in\{1, 2, 3\}$.
By symmetry between $z_1z_2$ and $z_3z_4$,
we may assume $b_1\notin \left[1, \frac{1}{2}k(k+1)\right]$.
\begin{claim}\label{preferred_pair}
There exists a distant pair $(i, j)$ with $b_2\notin I(i, j)$.
\end{claim}
\claimproof\ \ {}
Since $k\ge 6$,
$(1, 4)$,
$(2, 5)$,
$(3, 6)$ are all distant pairs,
and $I(1, 4)=\{k, 1, 3, 4\}$,
$I(2, 5)=\{1, 2, 4, 5\}$
and $I(3, 6)=\{2, 3, 5, 6\}$.
Since $k\ge 6$,
$k, 1, 2, 3, 4, 5, 6$ are distinct indices
except for the case of $k=6$.
However,
even if $k=6$,
$6$ does not appear in $I(2, 5)$.
Thus,
$I(1, 4)\cap I(2, 5)\cap I(3, 6)=\emptyset$,
and hence whatever value $b_2$ takes,
one of $I(1,4)$,
$I(2,5)$ and $I(3,6)$ excludes $b_2$.
\qed
\par
By Claim~\ref{preferred_pair},
we take a distant pair $(i, j)$
with $b_2\notin I(i, j)$.
Let $p=p(i, j)$ and $q=q(i, j)$.
Then $x_ix_j=x_px_q$,
$c(x_px_q)\in\{p-1, p\}$,
$c(x_px_q)\notin \{q, q-1\}$
and $b_2\notin I(p, q)=\{p-1, p, q-1, q\}$.
Moreover,
by the remark~(ii),
we may assume $c(x_px_q)=p$
by reversing the order of indices in $x_i$ if necessary,
\par
Note $K[z_2z_1, z_2z_3, z_2x_p, x_px_{p-1}, x_px_{p+1}]$
is a barbell edge-colored in $\{b_1, b_2, c(z_2x_p), p-1, p\}$.
Then since $b_2, p-1, p$ are distinct colors
and $b_1\notin \left[1, \frac{1}{2}k(k+1)\right]$,
we have $c(z_2x_p)\in\{b_2, p-1, p\}$.
However,
if $c(z_2x_p)=b_2$,
then $K[x_pz_2, x_px_{p-1}x_px_q, x_qx_{q-1}, x_qx_{q+1}]$
is a rainbow barbell edge-colored in
$\{b_2, p-1, p, q-1, q\}$.
a contradiction.
Hence we have $c(z_2x_p)\in \{p-1, p\}$.
On the other hand,
$K[z_2z_1, z_2z_3, z_2x_q, x_qx_{q-1}, x_qx_{q+1}]$,
$K[z_2z_1, z_2z_3, z_2x_q, x_qx_{q-1}, x_qx_p]$
and
$K[z_2z_1, z_2z_3, z_2x_q, x_qx_{q+1}, x_qx_p]$
are barbells,
and they are edge-colored in
$\{b_1, b_2, c(z_2x_q), q-1, q\}$,
$\{b_1, b_2, c(z_2x_q), q-1, p\}$ and
$\{b_1, b_2, c(z_2x_q), p, q\}$,
respectively.
Since none of them is rainbow,
we have $c(z_2x_q)=b_2$.
However,
then $K[z_2z_1, z_2x_p, z_2x_q, x_qx_{q-1}, x_qx_{q+1}]$
is a rainbow barbell
edge-colored in $\{b_1, c(z_2x_p), b_2, q-1, q\}$,
whether $c(z_2x_p)=p$ or $c(z_2x_p)=p-1$.
This is a final contradiction,
and (1) follows.
\par\noindent
(2)\ \ {}
Assume $B\le C_k$.
Then there exists a positive integer $t$ such that
every rainbow $B$-free edge-colored complete graph
colored in $t$ or more colors is rainbow $C_k$-free.
Take a set of $2t-k$ vertices
$X=\{x_1,\dots, x_k\}\cup\{y_{k+1},\dots, y_t, z_{k+1},\dots z_t\}$
and let $K=K[X]$.
Define $c\colon E(K)\to\naturalnumbers$ by
\[
c(e)=
\begin{cases}
i & \text{if $e=x_ix_{i+1}$, $1\le i\le k$}\\
j & \text{if $e=y_jz_j$, $k+1\le j\le t$}\\
1 & \text{otherwise,}
\end{cases}
\]
where we consider $x_{k+1}=x_1$.
Then $(K, c)$ is edge-colored in $t$ colors
and $C=x_1x_2\dots x_kx_1$ is a rainbow cycle of order $k$.
\par
We claim that $(K, c)$ is rainbow $B$-free.
Assume,
to the contrary,
that $(K, c)$ contains a rainbow subgraph $T$ which is isomorphic to
the barbell $B$.
Let $u$ and $v$ be the two adjacent vertices that correspond to
the vertices of degree~$3$ of $B$.
Then both $u$ and $v$ have color degree~$3$ in $(K, c)$.
On the other hand,
by the definition of $c$,
$\Delta_c(K, c)=3$ and
only the vertices $x_i$ with $3\le i\le k$ have color degree~$3$.
Moreover,
the color~$1$ commonly appears in the edges incident with $u$ and incident with $v$.
This is possible only if $c(u,v)=1$.
Then $T-uv$ consists of two vertex-disjoint paths of order~$3$ contained in $C$.
However,
since $k\le 5$,
this is impossible.
Therefore,
(2) follows.   
\end{proof}
\subsection{Tree vs Tree}
There are a broad range of combinations
in the relations among trees,
and we currently do not know much.
In $\HH$,
$K_{1,3}$ is the smallest tree and
its order is $4$.
The trees of order~$5$ is $P_5$,
$K^+_{1,3}$ and $K_{1,4}$.
We study these trees in detail in the next section.
Among trees of order~$6$,
we study the relations among $S_{2, 2, 1}$,
$S_{3, 1, 1}$ and $P_6$.
Actually,
they are pairwise incomparable.
Since $S_{2, 2, 1}\le C_4$ but
$S_{3, 1, 1}\not\le C_4$ and
$P_6\not\le C_4$,
we immediately have
$S_{3, 1, 1}\not\le S_{2,2,1}$ and
$P_6\not\le S_{2, 2, 1}$.
\begin{theorem}
\leavevmode
\vspace{-1ex}
\begin{enumerate}
\item
$S_{2, 2, 1}\not\le S_{3, 1, 1}$
\item
$P_6\not\le S_{3,1,1}$
\item
$S_{3, 1, 1}\not\le P_6$
\item
$S_{2,2,1}\not\le P_6$
\end{enumerate}
\label{P6C4S311S221}
\end{theorem}
\begin{proof}
(1)\ \ {}
Assume $S_{2, 2, 1}\le S_{3, 1, 1}$.
Then there exists a positive integer $t$ such that
every rainbow $S_{2, 2, 1}$-free edge-colored complete graph
colored in $t$ or more colors is rainbow $S_{3, 1, 1}$-free.
We may assume $t\ge 4$.
\par
Take $t+1$ vertices $x_0, x_1,\dots x_{t-2}, y, z$
and let $K=K[\{x_0, x_1,\dots, x_{t-2}, y, z\}]$.
Define $c\colon E(K)\to \naturalnumbers$ by
\[
c(e) = 
\begin{cases} 
i & \text{if $e=x_0x_i$, $1\le i\le t-2$}\\
t-1 & \text{if $e=yz$}\\
t & \text{otherwise.}
\end{cases}
\]
Then $(K, c)$ is colored in $t$ colors and
$K[x_0x_1, x_0x_2, x_0x_3, x_3y, yz]$ is a tree isomorphic to $S_{3, 1, 1}$.
Moreover it is colored in $\{1, 2, 3, t-1, t\}$ and hence it is rainbow.
On the other hand,
$x_0$ is the only vertex of color degree at least $3$ in $K$,
and every path of length~$2$ starting at $x_0$
contains an edge which is colored in $t$.
This implies that $(K, c)$ does not contain a rainbow spider
which has two legs of length~$2$,
and hence $K$ is rainbow $S_{2, 2, 1}$-free.
This is a contradiction.
\par\noindent
(2)\ \ {}
Assume $P_6\le S_{3,1,1}$.
Then there exists a positive integer $t$
such that every rainbow $P_6$-free
edge-colored complete graph colored in $t$ or more colors is rainbow $S_{3,1,1}$-free.
We may assume $t\ge 5$.
\par
Take a set of $t$ vertices $X=\{x_0, x_1,\dots, x_{t-2}, y\}$
and let $K=K[X]$.
Define $c\colon E(K)\to\naturalnumbers$ by
\[
c(e) = 
\begin{cases}
i & \text{if $e=x_0x_i$, $1\le i\le t-2$}\\
t-1 & \text{if $e=x_{t-2}y$}\\
t & \text{otherwise.}    
\end{cases}
\]
Then $(K, c)$ is edge-colored in $t$ colors.
Every path can contain at most $2$ edges incident with $x_0$
and all the edges not incident with $x_0$ are colored in $t-1$ or $t$.
This implies that every rainbow path in $(K, c)$ has length at most $4$,
and hence $(K, c)$ is rainbow $P_6$-free.
On the other hand,
$K[x_0x_{t-2}, x_{t-2}y, yx_3, x_0x_1, x_0x_2]$ is a rainbow subgraph
isomorphic to $S_{3, 1, 1}$ colored in $\{1, 2, t-2, t-1, t\}$.
This is a contradiction.
\par\noindent
(3)\ \ {}
Assume $S_{3,1,1}\le P_6$.
Then there exists a positive integer $t$ such
that every rainbow $S_{3,1,1}$-free edge-colored complete graph
colored in $t$ or more colors is
rainbow $P_6$-free.
We may assume $t\ge 5$.
\par
Take a set of $6+2(t-5)$ vertices
$X=\{x_1, x_2, x_3, x_4, x_5, x_6, y_6, \dots y_t, z_6, \dots, z_t\}$
and let $K=K[X]$.
Define $c\colon E(K)\to\naturalnumbers$ by
\[
c(e) = 
\begin{cases}
i & \text{if $e=x_ix_{i+1}$, $1\le i\le 5$}\\
j & \text{if $e=y_jz_j$, $6\le j\le t$}\\
2 & \text{otherwise.}
\end{cases}
\]
Then $(K, c)$ is colored in $t$ colors.
Assume $(K, c)$ contains a rainbow tree $T$ isomorphic to $S_{3,1,1}$.
Let $u$ be the unique vertex of degree~$3$ in $T$.
Since $\Delta_c(K, c)=3$ and
$x_4$ and $x_5$ are the only vertices of color degree~$3$ in $K$,
we have $u\in\{x_4, x_5\}$.
Moreover,
one of the edges incident with $u$ in $T$ is colored in~$2$.
Since $T$ is rainbow,
this implies $x_2x_3\notin E(T)$.
Therefore,
no leg of $T$ can contain three edges,
and hence $T$ cannot be isomorphic to $S_{3,1,1}$.
Therefore,
$(K, c)$ is rainbow $S_{3,1,1}$-free.
On the other hand,
$x_1x_2x_3x_4x_5x_6$ is a rainbow path of order~$6$ in
$(K, c)$.
This is a contradiction.
\par\noindent
(4)\ \ {}
Assume $S_{2,2,1}\le P_6$.
Then there exists a positive integer $t$
such that every rainbow $S_{2,2,1}$-free edge-colored complete graph
colored in $t$ or more colors is rainbow $P_6$-free.
We may assume $t\ge 5$.
\par
Take a set of $t+1$ vertices $X=\{x_1, x_2, x_3, x_4, x_5, x_6, y_6, y_7,\dots, y_t\}$
and let $K=K[X]$.
Define $c\colon E(K)\to\naturalnumbers$ by
\[
c(e) = 
\begin{cases}
i & \text{if $e=x_ix_{i+1}$, $1\le i\le 5$}\\
j & \text{if $e=x_2y_j$, $6\le j\le t$}\\
3 & \text{otherwise.}
\end{cases}
\]
Then $(K, c)$ is colored in $t$ colors.
\par
If we remove all the edges colored in~$3$ in $(K, c)$,
the resulting graph consists of two components
isomorphic to $K_{1, t-3}$ and $P_3$.
Since both are stars,
it is impossible to obtain a spider having two legs of length~$2$
by adding just one edge.
This implies that no rainbow subgraph
of $(K, c)$ can be isomorphic to $S_{2, 2, 1}$,
whether it contains an edge colored in $3$ or not.
Therefore,
$(K, c)$ is rainbow $S_{2,2,1}$-free.
On the other hand $x_1x_2x_3x_4x_5x_6$ is a rainbow path of order~$6$ in $(K, c)$.
This is a contradiction.
\end{proof}
%%%%%%%%%%%%%%%%%%%%%%%%%%%%%%%%%%%%%%%%%%%%%%%%%%%%%%%%%%%%%%
\section{Structure of $\boldsymbol{(\qHH, \le\,)}$}
As we mentioned in Section~2,
for $H_1, H_2\in\HH$,
we write $H_1\equiv H_2$ if
both $H_1\le H_2$ and $H_2\le H_1$ hold.
Moreover,
for two equivalence classes $\HH_1$ and $\HH_2$
with respect to $\equiv$,
we write $\HH_1\le \HH_2$
if $H_1\le H_2$ holds
for some $H_1\in\HH_1$ and $H_2\in\HH_2$.
Then $(\qHH, \le\,)$ is a partially ordered set.
In this section,
we investigate its structure.
First,
we prove that
it has the minimum element.
\begin{theorem}\label{minimun_elements}
\ \ {}
\begin{enumerate}
\item
For each $H\in\HH$,
$K_{1,3}\le H$ holds.
\item
If $H\in\HH$ and $H\le K_{1,3}$,
then either $H=K_{1,3}$ or $K_{1,3}^+$.
In particular,
$\{K_{1,3}, K^+_{1,3}\}$ is the minimum element of $(\qHH, \le\,)$.
\end{enumerate}
\end{theorem}
Before proving this theorem,
we prove one lemma.
For an integer $k$ with $k\ge 2$,
let $K_{1,k}+e$ be the graph
obtained from $K_{1,k}$ by adding one edge
between a pair of its endvertices.
Note $K_{1,2}+e=C_3$
and $K_{1,3}+e=Z_1$.
\begin{lemma}
For every integer $k$ with $k\ge 2$,
$P_5\not\le K_{1,k}+e$.
In particular,
$P_5\not\le K_{1,k}$.
\label{P_5_not}
\end{lemma}
\begin{proof}
Assume $P_5\le K_{1,k}+e$ for some integer $k$ with $k\ge 2$.
Then there exists a positive integer $t$ such that every rainbow $P_5$-free
complete graph $(K, c)$ edge-colored in $t$ or more colors
is rainbow $(K_{1,k}+e)$-free.
We may assume $t\ge k+1$.
\par
Take $t$ vertices $x_0, x_1,\dots, x_{t-1}$ and let $K=K[{\{x_0, x_1,\dots, x_{t-1}\}}]$.
Define $c\colon E(G)\to\{1, 2,\dots, t\}$ by
\[
c(e) =
\begin{cases}
i & \text{if $e=x_0x_i$, $1\le i\le t-1$,}\\
t & \text{otherwise.}
\end{cases}
\]
Then $(K, c)$ is colored in $t$ colors.
Also every rainbow path in $(K, c)$ can contain at most two edges incident with $x_0$ and
hence it can contain at most three edges altogether.
Therefore,
$(K, c)$ is rainbow $P_5$-free.
On the other hand,
$K[\{x_0x_1, x_0x_2,\dots, x_0x_k, x_1x_2\}]$
is a rainbow subgraph in $(K, c)$ isomorphic to $K_{1,k}+e$
colored in
$\{1, 2,\dots, k, t\}$.
This is a contradiction.
\par
Note $K_{1,k}\subseteq K_{1,k}+e$ and hence $K_{1,k}\le K_{1,k}+e$.
Therefore,
if $P_5\le K_{1,k}$,
the transitivity of $\le$ yields $P_5\le K_{1,k}+e$,
a contradiction.
Thus,
we have $P_5\not\le K_{1,k}$.
\end{proof}
\begin{proof}[\textbf{Proof of Theorem~\ref{minimun_elements}.}]
(1)\ \ {}
If $\Delta(H)\ge 3$,
then $K_{1,3}\subseteq H$,
which implies $K_{1,3}\le H$.
If $\Delta(H)\le 1$,
then $H$ is a subgraph of $P_4$,
which contradicts the definition of $\HH$.
Therefore,
we may assume $\Delta(H)=2$,
which means that
$H$ is either a path or a cycle.
Since $P_5\subseteq P_k$ and hence $P_5\le P_k$ for $k\ge 5$,
and $C_3\le C_k$ for $k\ge 4$ by Corollary~\ref{C3Ck},
it suffices to prove $K_{1,3}\le P_5$ and $K_{1,3}\le C_3$.
\par
First,
we claim that every rainbow $K_{1,3}$-free edge-colored complete graph
$(K, c)$ is rainbow $P_5$-free.
Assume,
to the contrary,
that $(K, c)$ contains a rainbow path
$P=x_1x_2x_3x_4x_5$ of order~$5$.
We may assume $c(x_1x_2)=1$,
$c(x_2x_3)=2$,
$c(x_3x_4)=3$
and $c(x_4x_5)=4$.
Since $\{x_1x_2, x_2x_3, x_2x_4\}$ induces $K_{1,3}$,
we have $c(x_2x_4)\in\{1, 2\}$.
Then $\{x_3x_4, x_4x_5, x_2x_4\}$
induces a rainbow $K_{1,3}$.
This is a contradiction.
\par
Next,
we claim that every rainbow $K_{1,3}$-free edge-colored complete graph
$(K, c)$ colored in $4$ or more colors
is rainbow $C_3$-free.
Assume,
to the contrary,
that $(K, c)$ contains a rainbow cycle $C=x_1x_2x_3x_1$.
We may assume $c(x_1x_2)=1$,
$c(x_2x_3)=2$ and $c(x_1x_3)=3$.
Note that since $K$ is edge-colored in $4$ or more colors,
$V(K)-V(C)\ne\emptyset$.
If $E_G\bigl( V(C), V(K)-V(C)\bigr)$ contains an edge $e$ with $c(e)\notin \{1, 2, 3\}$,
then this edge and the two edges of $C$ incident with the endvertex of $e$
induce a rainbow $K_{1,3}$,
a contradiction.
Therefore,
every edge in $E_G\bigl(V(C), V(G)-V(C)\bigr)$ is colored in $1$,~$2$
or~$3$.
\par
Take an edge $f=uv$ with $c(f)\notin\{1, 2, 3\}$.
Then $\{u, v\}\cap V(C)=\emptyset$.
If two of the edges $ux_1$,
$ux_2$ and $ux_3$ receive different colors,
then these edges together with $f$
induce a rainbow $K_{1,3}$ with center $u$,
a contradiction.
Hence $ux_1$,
$ux_2$ and $ux_3$ all receive the same color.
We may assume $c(ux_1)=c(ux_2)=c(ux_3)=1$.
Then $\{x_1x_3, x_2x_3, x_2u\}$ induces a rainbow $K_{1,3}$
edge-colored in $\{1, 2, 3\}$.
This is a contradiction.
\smallbreak\noindent
(2)\ \ {}
Assume $H\le K_{1,3}$ and that $H$ is neither $K_{1,3}$ nor $K^+_{1,3}$.
If $H\subsetneq K_{1,3}$,
then $H$ is a subgraph of $P_4$,
which contradicts the hypothesis.
If $K_{1,3}\subseteq H$,
then either $H=K_{1,3}$ or $K^+_{1,3}$ by Theorem~\ref{clms}.
Therefore,
$H$ is neither a subgraph or a supergraph of $K_{1,3}$.
This implies that $\Delta(H)=2$.
Moreover,
since $K_{1,3}$ is a tree,
$H$ is also a tree by Theorem~\ref{dim_monotone}.
Therefore,
we see that $H$ is a path
of order at least~$5$.
However,
we have $P_5\not\le K_{1,3}$ by Lemma~\ref{P_5_not}
and hence $H\not\le K_{1,3}$.
This is a contradiction.
\end{proof}
Next,
we determine
the elements of $(\qHH, \le\,)$ immediately after the minimum element.
They are $\{P_5\}$,
$\{C_3\}$ and $\{K_{1,4}, K^+_{1,4}\}$.
Before stating this result,
we prove one lemma
concerning the relation between $P_5$ and the barbell.
\begin{lemma}\label{barbell}
$P_5\le B$
\end{lemma}
\begin{proof}
We prove that every $P_5$-free edge-colored complete graph $(K, c)$ is rainbow $B$-free.
Assume $(K, c)$ contains a rainbow subgraph $G$ which
is isomorphic to $B$.
Let $V(G)=\{x_1, x_2, x_3, x_4, x_5, x_6\}$
and $E(G)=\{x_1x_3, x_2x_3, x_3x_4, x_4x_5, x_4x_6\}$.
We may also assume
$c(x_1x_3)=1$,
$c(x_2x_3)=2$,
$c(x_3x_4)=3$,
$c(x_4x_5)=4$ and $c(x_4x_6)=5$.
\par
Since $x_1x_2x_3x_4x_5$ is a $5$-path and $(K, c)$ is rainbow $P_5$-free,
$c(x_1x_2)\in\{2, 3, 4\}$.
Also,
by considering
$x_2x_1x_3x_4x_6$,
we have $c(x_1x_2)\in \{1, 3, 5\}$.
Hence  we have $c(x_1x_2)=3$.
By symmetry,
we also have $c(x_5x_6)=3$.
\par
Consider $x_3x_2x_1x_4x_5$ and $x_3x_2x_1x_4x_6$.
Then we have $c(x_1x_4)\in\{2, 3, 4\}$ and $c(x_1x_4)\in\{2, 3, 5\}$.
Thus,
we have $c(x_1x_4)\in \{2, 3\}$.
However,
if $c(x_1x_4)=2$,
then $x_3x_1x_4x_5x_6$ is a rainbow path colored in $\{1, 2, 3, 4\}$,
and if $c(x_1x_4)=3$,
then $x_2x_3x_1x_4x_5$ is a rainbow path colored in
$\{1, 2, 3, 4\}$.
Therefore,
we obtain a contradiction in both cases.
\end{proof}
Now we prove the following theorem.
\begin{theorem}
Let $H\in\HH$ and $H\notin \{K_{1, 3}, K^+_{1, 3}\}$.
Then
\begin{enumerate}
\item
$K_{1,4}\le H$,
$P_5\le H$ or $C_3\le H$.
Moreover,
\item
if $H\le P_5$,
then $H=P_5$,
\item
if $H\le C_3$,
then $H=C_3$,
and
\item
if $H\le K_{1,4}$,
then either $H=K_{1,4}$ or $H=K^+_{1,4}$.
\end{enumerate}
\label{next_minimum}
\end{theorem}
\begin{proof}
(1)\ \ {}
Assume,
to the contrary,
that $K_{1,4}\not\le H$,
$P_5\not\le H$ and $C_3\not\le H$.
Then $K_{1,4}\not\subseteq H$,
$P_5\not\subseteq H$ and
$C_3\not\subseteq H$.
\par
Since $K_{1,4}\not\subseteq H$,
$\Delta(H)\le 3$.
On the other hand,
if $\Delta(H)\le 2$,
then $H$ is either a path or a cycle.
However,
since $H\in\HH$ and $P_5\not\subseteq H$,
$H$ is not a path.
Moreover,
a cycle of order~$5$ or more contains $P_5$ as a subgraph,
while $P_5\le C_4$ by Corollary~\ref{P5C4} and $C_3\not\subseteq H$ forces $H\ne C_3$.
Therefore,
$H$ is not a cycle.
This is a contradiction.
Hence we have $\Delta(H)=3$.
\par
Let $v_1$ be a vertex of degree~$3$ in $H$.
Let $N_H(v_1)=\{v_2, v_3, v_4\}$.
Since $C_3\not\subseteq H$,
$\{v_2, v_3, v_4\}$ is an independent set.
Then since $H\ne K_{1,3}$
and $H$ is connected,
$H$ contains a vertex $v_5$ with $N_H(v_5)\cap\{v_2, v_3, v_4\}\ne\emptyset$.
By symmetry,
we may assume $v_2v_5\in E(H)$.
\par
If $N_H(v_5)\cap \{v_3, v_4\}\ne\emptyset$,
then $H$ contains $C_4$.
However,
this implies $C_4\le H$.
Since $P_5\le C_4$,
we have $P_5\le H$.
This contradicts the assumption.
Therefore,
we have $N_H(v_5)\cap\{v_3, v_4\}=\emptyset$.
\par
Since $H\ne K_{1,3}^+$,
$V(H)\ne\{v_1, v_2, v_3, v_4, v_5\}$.
Then since $H$ is connected,
$H$ contains a vertex $v_6$
which is adjacent to $x_2$,
$x_3$,
$x_4$ or $x_5$.
However,
if $v_6$ is adjacent to either $v_3$,
$v_4$ or $v_5$,
$H$ contains $P_5$,
a contradiction.
Therefore,
the only possibility is $v_2v_6\in E(G)$.
However,
in this case we have $B\subset H$ and hence $B\le H$.
Since $P_5\le B$ by Lemma~\ref{barbell},
we have $P_5\le H$,
a contradiction.
Therefore,
(1) follows.
\par\noindent
(2)\ \ {}
Assume $H\ne P_5$.
Then since $H\in \HH$,
$H$ is not a subgraph of $P_5$.
Also,
since $H\le P_5$,
$H$ is a tree
by Theorem~\ref{dim_monotone}.
Again since $H\le P_5$,
there exists a positive integer $t$ such that every rainbow $H$-free
edge-colored complete graph $(K, c)$ colored in $t$ or more colors
is rainbow $P_5$-free.
We may assume $t\ge 4$.
Introduce $5+2(t-4)$ vertices,
$x_1$,
$x_2$,
$x_3$,
$x_4$,
$x_5$,
$y_5,\dots, y_t$,
$z_5,\dots, z_t$.
Let $X=\{x_1, x_2, x_3, x_4, x_5\}$ and
$Y=\{y_5,\dots, y_t, z_5,\dots, z_t\}$,
and let $K=K[X\cup Y]$.
Define $c\colon E(G)\to\{1,\dots, t\}$ by
\[
c(e) = 
\begin{cases}
i & \text{if $e=x_ix_{i+1}$, $1\le i\le 4$,}\\
j & \text{if $e=y_jz_j$, $5\le j\le t$,}\\
2 & \text{otherwise.}    
\end{cases}
\]
Let $P=x_1x_2x_3x_4x_5$.
Then $(K, c)$ is edge-colored in $t$ colors and $P$ is a rainbow path of order~$5$.
Therefore,
$(K, c)$ contains a rainbow subgraph $G$ which is isomorphic to $H$.
Let $F=\bigl(E(K[X])-\{x_1x_2, x_2x_3, x_3x_4, x_4x_5\}\bigr)\cup E_K(X, Y)$.
Note that every edge in $F$ is colored in~$2$.
\par
Since a maximal connected rainbow subgraph of $G-X$ is isomorphic to $P_4$,
we have $V(G)\cap X\ne\emptyset$.
On the other hand,
$H$ is connected and it is not a subgraph of $P_5$,
$E(G)\cap F\ne\emptyset$.
Let $e\in E(G)\cap F$.
Then since $c(e)=2$ and $G$ is rainbow,
$x_2x_3\notin E(G)$ and
$E(G)\cap F=\{e\}$.
Suppose $e\in E_K(X, Y)$.
We may assume $e=uy_5$,
where $u\in X$.
Then $G$ is a subgraph of  $K[x_1x_2, uy_5, y_5z_5]$
or $K[x_3x_4, x_4x_5, uy_5, y_5z_5]$.
However,
in the former case,
$H$ is a subgraph of $P_4$,
and in the latter case,
$H$ is a subgraph of $P_5$ or $K_{1,3}^+$.
Hence we obtain a contradiction in either case.
\par
Next,
suppose $e\in E(K[X])$.
In this case,
$P-x_2x_3+e$ is isomorphic to either $P_5$ or $K_{1,3}^+$,
and $H$ is its subgraph.
This is again a contradiction.
\par\noindent
(3)\ \ {}
Assume $H\ne C_3$.
If $H$ contains a cycle of order~$3$,
then $C_3\subseteq H$.
However,
by Theorem~\ref{clms},
not both $C_3\subsetneq H$ and $H\le C_3$ occur simultaneously.
This is a contradiction.
Therefore,
$H$ does not contain a cycle of order~$3$.
\par
Since $H\le C_3$,
there exists a positive integer $t$ such that every rainbow $H$-free
edge-colored complete graph $(K, c)$ colored in $t$ or more colors
is rainbow $C_3$-free.
We may assume $t\ge 3$.
\par
Introduce a set of $t$ vertices $X=\{x_0, x_1,\dots x_{t-1}\}$ and let $K=K[X]$.
Define $c\colon E(K)\to\naturalnumbers$ by
\[
c(e) =
\begin{cases}
i & \text{if $e=x_0x_{i}$, $1\le i\le t-1$}\\
t & \text{otherwise.}
\end{cases}
\]
Then $(K, c)$ is edge-colored in $t$ colors
and $x_0x_1x_2x_0$ is a rainbow cycle of order~$3$
colored in $\{1, 2, t\}$.
Therefore,
by the assumption,
$(K, c)$ contains a rainbow subgraph $G$
which is isomorphic to $H$.
On the other hand,
in $(K, c)$,
every rainbow cycle has order~$3$.
This implies that $H$ does not contain a cycle and hence
it is a tree.
\par
Every rainbow tree in $(K, c)$ is isomorphic to
$K_{1, k}$ or $K^+_{1, k}$ for some $k\ge 2$.
On the other hand,
since $C_3$ is $2$-regular and $H\le C_3$,
we have $\Delta(H)\le 3$ by Theorem~\ref{invariants}~(4).
Therefore,
we have $H=K_{1,3}$ or $H=K^+_{1,3}$.
However,
this contradicts the hypothesis.
\par\noindent
(4)\ \ {}
Assume $H\le K_{1,4}$ and
$H\notin\{K_{1,3}, K^+_{3,1}, K_{1,4}, K^+_{1,4}\}$.
By Theorem~\ref{dim_monotone},
$H$ is a tree.
If $\Delta(H)\ge 4$,
then $K_{1,4}\subseteq H$.
However,
this yields $H\in\{K_{1,4}, K^+_{1,4}\}$ by Theorem~\ref{clms},
a contradiction.
Therefore,
we have $\Delta(H)\le 3$. 
\par
Since $H\le K_{1,4}$,
there exists a positive integer $t$ such that
every rainbow $H$-free edge-colored complete graph
colored in $t$ or more colors is rainbow $K_{1,4}$-free.
We may assume $t\ge 4$.
Take a set of $t+1$ vertices $X=\{x_0, x_1,\dots, x_t\}$ and let
$K=K[X]$.
Define $c\colon E(K)\to \naturalnumbers$ by
\[
c(e) = 
\begin{cases}
i & \text{if $e=x_0x_i$, $1\le i\le t$}\\
1 & \text{otherwise.}
\end{cases}
\]
Then $(K, c)$ is colored in $t$ colors and
$K[x_0x_1, x_0x_2, x_0x_3, x_0x_4]$ is
a rainbow star of order~$5$.
Therefore,
$(K, c)$ contains a rainbow subgraph $G$
which is isomorphic to $H$.
However,
since every rainbow tree in $(K, c)$
is isomorphic to $K_{1, k}$ or $K^+_{1, k}$ for some $k\ge 2$
and $\Delta(H)\le 3$,
we have $H\in \{K_{1,3}, K^+_{1,3}\}$.
This is a contradiction.
\end{proof}
We next investigate stars.
Since $K_{1,k}\subseteq K_{1,k+1}$,
we have $K_{1, k}\le K_{1, k+1}$.
Also,
Theorem~\ref{clms} says $K_{1,k}\equiv K^+_{1,k}$
and that $K^+_{1,k}$ is the only proper supergraph
of $K_{1,k}$ that is equivalent with $K_{1,k}$.
Here we prove that
even if we search the entire $\HH$,
which contains graphs not comparable with
$K_{1,k}$ with respect to $\subseteq$,
$K^+_{1,k}$ is still the only graph equivalent with $K_{1,k}$.
We also prove that $\{K_{1,k+1}, K^+_{1,k+1}\}$ is an element immediately after
$\{K_{1,k}, K^+_{1,k}\}$.
\begin{theorem}
Let $H\in\HH$
and let $k_1$ and $k_2$ be integers
with $3\le k_1\le k_2$.
If both $K_{1, k_1}\le H$ and $H\le K_{1, k_2}$ hold,
then $H=K_{1,k}$ or $H=K_{1,k}^+$
for some $k$ with $k_1\le k\le k_2$.
\label{stars}
\end{theorem}
\begin{proof}
Note that since $H\le K_{1,k_2}$,
$H$ is a tree by Theorem~\ref{dim_monotone}.
\par
Since $H\le K_{1, k_2}$,
there exists a positive integer $t$ such that
every rainbow $H$-free edge-colored complete graph $(K, c)$
colored in $t$ or more colors is rainbow $K_{1, k_2}$-free.
We may assume $t\ge k_2+1$.
Introduce a set of $t$ vertices
$X=\{x_0, x_1,\dots, x_{t-1}\}$ and let
$K=K[X]$.
Define $c\colon E(K)\to\{1, 2,\dots, t\}$ by
\[
c(e)=
\begin{cases}
i & \text{if $e=x_0x_i\quad 1\le i\le t-1$}\\
t & \text{otherwise.}
\end{cases}
\]
Then $(K, c)$ is a complete graph edge-colored in $t$ colors,
and $K[x_0x_1, x_0x_2,\dots, x_0x_{k_2}]$
is a rainbow subgraph isomorphic to $K_{1, k_2}$.
This implies that $(K, c)$ contains a rainbow tree $T$
which is isomorphic to $H$.
\par
Since $T$ is rainbow,
$T$ contains at most one edge colored in $t$.
Since $T$ is a tree,
this yields that $H=K_{1,k}$ or $H=K_{1,k}^+$ for some $k$
with $3\le k\le t-1$.
In particular,
$H\le K_{1,k}$.
\par
If $k > k_2$,
then $K_{1,k_2}\subseteq H$.
This together with the hypothesis of $H\le K_{1, k_2}$
yields $k=k_2$ by Theorem~\ref{clms}.
This is a contradiction,
and hence we have $k\le k_2$.
\par
Assume $k < k_1$.
Then $K_{1,k}\subseteq K_{1, k_1}$.
On the other hand,
since $K_{1,k_1}\le H$ and $H\le K_{1,k}$,
we have $K_{1,k_1}\le K_{1,k}$.
These yield $k=k_1$ by Theorem~\ref{clms},
a contradiction.
Therefore,
we have $k_1\le k\le k_2$.
 by Theorem~\ref{clms}. 
\end{proof}
By setting $k_2=k_1$ and $k_2=k_1+1$ in theorem~\ref{stars},
we obtain the following corollary.
\begin{corollary}
For each integer $k$ with $k\ge 3$,
The equivalence class with respect to $\equiv$ in $\HH$ containing $K_{1,k}$
is $\{K_{1,k}, K_{1,k}^+\}$.
Moreover,
$\{K_{1, k+1}, K^+_{1, k+1}\}$ is immediately after
$\{K_{1,k}, K^+_{1, k}\}$ in $(\qHH, \le\,)$.
\end{corollary}
%%%%%%%%%%%%%%%%%%%%%%%%%%%%%%%%%%%%%%%%%%%%%%%%%%%%%%%%%%%%%%%%%%%%%%%%%%
\section{Concluding Remarks}
In this paper,
we have investigated a binary relation $\le$,
which gives a preorder among rainbow forbidden subgraphs.
The previous study~\cite{CLMS}
shows that there are not many kinds of pairs $(H_1, H_2)$ with
$H_1\le H_2$
under the condition that $H_1$ or $H_2$ is a subgraph of the other.
However,
By lifting this condition,
we have found a large variety of pairs.
Also,
we have introduced an equivalence relation $\equiv$ naturally defined from $\le$,
and converted $\le$ to a partial order on its equivalence classes.
We have studied the properties of this partial order.
\par
Our research only determines a small part of the partially ordered set $(\qHH,\le\,)$.
In particular,
though we have proved that $\{K_{1,k}, K^+_{1,k}\}$ is an equivalence class
in $\HH$,
we have not found any other non-singleton equivalence class.
Thus,
we raise the following problem.
\begin{problem}
Is there a non-singleton equivalence class of $\equiv$
other than $\{K_{1,k}, K^+_{1,k}\}$?
\end{problem}
\par
We have proved that $(\qHH, \le\,)$ has the minimum element,
and determined the elements immediately after it.
However,
we do not think it gives sufficient information to
grasp a large picture on the structure of $\qHH$.
As the next step,
we place the next question at the end of this paper.
\begin{problem}
Determine the elements immediately after
$\{C_3\}$,
$\{P_5\}$ and
$\{K_{1,4}, K^+_{1,4}\}$.
\end{problem}

\end{document}